\documentclass[12pt,reqno]{amsart}
\usepackage{amsmath} 

\usepackage{amsthm}
\usepackage{amssymb}
\usepackage{graphics}
\usepackage{latexsym}
\usepackage{comment}

\numberwithin{equation}{section}
\newtheorem{thm}{Theorem}[section]
\newtheorem{prop}[thm]{Proposition}
\newtheorem{lemm}[thm]{Lemma}
\newtheorem{cor}[thm]{Corollary}

\theoremstyle{remark}
\newtheorem{rem}{Remark}[section]
\newtheorem{defn}{Definition}

\newcommand{\BBB}{\mathbb}
\newcommand{\R}{{\BBB R}}
\newcommand{\Z}{{\BBB Z}}

\newcommand{\N}{{\BBB N}}
\newcommand{\C}{{\BBB C}}

\newcommand{\ee}{\mbox{\boldmath $1$}}

\newcommand{\LR}[1]{{\langle {#1} \rangle }}

\newcommand{\cross}{\times}

\newcommand{\al}{\alpha}
\newcommand{\be}{\beta}
\newcommand{\ga}{\gamma}

\newcommand{\e}{\varepsilon}
\newcommand{\ta}{\tau}

\newcommand{\p}{\partial}

\newcommand{\supp}{\operatorname{supp}}

\newcommand{\kuuhaku}{\text{}}

\newcommand{\F}{\mathcal{F}}

\title[Sharp bilinear estimates for system of qDNLS]{Sharp bilinear estimates 
and its application to a system of quadratic derivative
nonlinear Schr\"odinger equations
}

\author[H. Hirayama]{Hiroyuki Hirayama}
\address[H. Hirayama]{Organization for Promotion of Tenure Track, University of Miyazaki, 1-1, Gakuenkibanadai-nishi, Miyazaki, 889-2192 Japan}
\email[H. Hirayama]{h.hirayama@cc.miyazaki-u.ac.jp}

\author[S. Kinoshita]{Shinya Kinoshita}
\address[S. Kinoshita]{Graduate School of Mathematics, Nagoya University,
Chikusa-ku, Nagoya, 464-8602, Japan}
\email[S. Kinoshita]{m12018b@math.nagoya-u.ac.jp}


\begin{document}
\begin{abstract}
In the present paper, 
we consider the Cauchy problem of the system of quadratic derivative nonlinear 
Schr\"odinger equations for the spatial dimension $d=2$ and $3$. 
This system was introduced by M. Colin and T. Colin (2004). 
The first author obtained some well-posedness 
results in the Sobolev space $H^{s}$. 
But under some condition for the coefficient of Laplacian, 
this result is not optimal.  
We improve the bilinear estimate
by using the nonlinear version of the classical Loomis-Whitney inequality,  
and prove the well-posedness in $H^s$ for $s\ge 1/2$ if $d=2$, 
and $s>1/2$ if $d=3$.
\end{abstract}
\maketitle
\setcounter{page}{001}


\section{Introduction\label{intro}}
We consider the Cauchy problem of the system of nonlinear Schr\"odinger equations:
\begin{equation}\label{NLS_sys}
\begin{cases}
\displaystyle (i\p_{t}+\alpha \Delta )u=-(\nabla \cdot w )v,
\ \ t>0,\ x\in \R^d,\\
\displaystyle (i\p_{t}+\beta \Delta )v=-(\nabla \cdot \overline{w})u,
\ \ t>0,\ x\in \R^d,\\
\displaystyle (i\p_{t}+\gamma \Delta )w =\nabla (u\cdot \overline{v}),
\ \ t>0,\ x\in \R^d,\\
(u(0,x), v(0,x), w(0,x))=(u_{0}(x), v_{0}(x), w_{0}(x)),\hspace{2ex}x\in \R^{d},
\end{cases}
\end{equation}
where $\al$, $\be$, $\ga\in \R\backslash \{0\}$, $d=2$ or $3$, 
and the unknown functions $u$, $v$, $w$ are $d$-dimensional complex vector valued. 
The system (\ref{NLS_sys}) was introduced by Colin and Colin in \cite{CC04} 
as a model of laser-plasma interaction. (See, also \cite{CC06}, \cite{CCO09_1}.)
They also showed that the 
local existence of the solution of (\ref{NLS_sys}) in $H^s$ for $s>d/2+3$. 
The system (\ref{NLS_sys}) is invariant under the following scaling transformation:
\begin{equation}\label{scaling_tr}
A_{\lambda}(t,x)=\lambda^{-1}A(\lambda^{-2}t,\lambda^{-1}x)\ \ (A=(u,v,w) ), 
\end{equation}
and the scaling critical regularity is $s_{c}=d/2-1$. 
We put
\begin{equation}\label{coeff}
\theta :=\alpha\beta\gamma \left(\frac{1}{\alpha}-\frac{1}{\beta}-\frac{1}{\gamma}\right),\ \ 
\kappa :=(\alpha -\beta)(\alpha -\gamma)(\beta +\gamma). 
\end{equation}
We note that $\kappa =0$ does not occur when $\theta > 0$. 

First, we introduce some known results for related problems. 
The system (\ref{NLS_sys}) has quadratic nonlinear terms which contains a derivative. 
A derivative loss arising from the nonlinearity makes the problem difficult. 
In fact, Mizohata (\cite{Mi85}) considered the Schr\"odinger equation
\[
\begin{cases}
i\partial_{t}u-\Delta u=(b_{1}(x)\cdot \nabla ) u,\ t\in \R ,\ x\in \R^{d},\\
u(0,x)=u_{0}(x),\ x\in \R^{d}
\end{cases}
\]
and proved that the uniform bound
\[
\sup_{x\in \R^{n},\omega \in S^{n-1},R>0}\left| {\rm Re}\int_{0}^{R}b_{1}(x+r\omega )\cdot \omega dr\right| <\infty.
\]
is necessary condition for the $L^{2}$ well-posedness. 
Furthermore, Christ (\cite{Ch}) proved that the flow map of 
the nonlinear Schr\"odinger equation
\begin{equation}\label{1dqdnls}
\begin{cases}
i\partial_{t}u-\partial_{x}^{2}u=u\partial_{x}u,\ t\in \R,\ x\in \R,\\
u(0,x)=u_{0}(x),\ x\in \R
\end{cases}
\end{equation}
is not continuous on $H^{s}$ for any $s\in \R$. 
From these results, 
it is difficult to obtain the well-posedness 
for quadratic derivative nonlinear Schr\"odinger equation in general. 
While for the system of quadratic derivative nonlinear equation, 
it is known that the well-posedness holds. 
In \cite{Hi}, the first author proved that
(\ref{NLS_sys}) with $d=2$ or $3$ is well-posed in $H^s$ for $s\ge s_c$
if $\theta >0$, and for $s\ge 1$ if $\theta \le 0$ and $\kappa \ne 0$. 
The first author also proved that 
the flow map is not $C^2$ for $s<1$ if $\theta = 0$ and $\kappa \ne 0$, 
and for $s<1/2$ if $\theta < 0$ and $\kappa \ne 0$. 
It says that
there is a gap of the regularity between the well-posedness and $C^2$-ill-posedness
under the condition $\theta < 0$ and $\kappa \ne 0$. 
The aim of this paper is to filling this gap.  
The well-posedness for $d=1$ and $d\ge 4$ are also obtained in \cite{Hi}. 
(See, Table~\ref{WP_NLS_sys} below.) 
\begin{table}[h]
\begin{center}
\begin{tabular}{|l|l|l|l|l|}
\hline
\multicolumn{2}{|c|}{} & $d=1$ & $d=2,3$ & $d\geq 4$\\
\hline
\multicolumn{2}{|c|}{$\theta > 0$} & \multicolumn{1}{|c|}{WP for $s\geq 0$} &\multicolumn{1}{|c|}{WP for $s\geq s_{c}$} & \multicolumn{1}{|c|}{WP for $s\geq s_{c}$}\\              
\cline{1-4}
$\kappa \ne 0$ & $\theta =0$ & \multicolumn{1}{|c|}{WP for $s\geq 1$} & 
WP for $s\geq 1$ & \multicolumn{1}{|c|}{}\\
\cline{2-3}
&$\theta < 0$ & \multicolumn{1}{|c|}{WP for $s\geq 1/2$} & &\multicolumn{1}{|c|}{}\\
\hline
\end{tabular}
\caption{Well-posedness (WP for short)  for (\ref{NLS_sys}) proved in \cite{Hi}}
\end{center}
\end{table}\label{WP_NLS_sys}

We point out that the results in \cite{Hi} does not 
contain the scattering of solution for $d\le 3$ under the condition $\theta =0$ 
(and also $\theta <0$). 
In \cite{IKS13}, Ikeda, Katayama, and Sunagawa considered 
the system of quadratic nonlinear Schr\"odinger equations
\begin{equation}\label{qDNLS}
\left(i\partial_t+\frac{1}{2m_j}\Delta\right)u_j=F_j(u,\partial_xu),\ \ t>0,\ x\in \R^d,\ j=1,2,3, 
\end{equation}
under the mass resonance condition 
$m_1+m_2=m_3$ (which corresponds to the condition $\theta =0$ for (\ref{NLS_sys})), 
where $u=(u_1,u_2,u_3)$ is $\C^3$-valued, 
$m_1$, $m_2$, $m_3\in \R\backslash \{0\}$, and $F_j$ is defined by
\begin{equation}\label{qDNLS_nonlin}
\begin{cases}
F_{1}(u,\partial_xu)=\sum_{|\alpha |, |\beta|\le 1}
C_{1,\alpha,\beta}(\overline{\partial^{\alpha}u_2})(\partial^{\beta}u_3),\\
F_{2}(u,\partial_xu)=\sum_{|\alpha |, |\beta|\le 1}
C_{1,\alpha,\beta}(\partial^{\beta}u_3)(\overline{\partial^{\alpha}u_1}),\\
F_{3}(u,\partial_xu)=\sum_{|\alpha |, |\beta|\le 1}
C_{1,\alpha,\beta}(\partial^{\alpha}u_1)(\partial^{\beta}u_2)
\end{cases}
\end{equation}
with some constants $C_{1,\alpha,\beta}$, $C_{2,\alpha,\beta}$, $C_{3,\alpha,\beta}\in \C$. 
They obtained the small data global existence and the scattering 
of the solution to (\ref{qDNLS})
in the weighted Sobolev space for $d=2$ 
under the mass resonance condition
and the null condition for the nonlinear terms (\ref{qDNLS_nonlin}). 
They also proved the same result for $d\ge 3$ without the null condition. 
In \cite{IKO16}, Ikeda, Kishimoto, and Okamoto proved
the small data global well-posedness and the scattering of the solution 
to (\ref{qDNLS}) in $H^s$ for $d\ge 3$ and $s\ge s_c$ 
under the mass resonance condition
and the null condition for the nonlinear terms (\ref{qDNLS_nonlin}). 
They also proved the local well-posedness in $H^s$
for $d=1$ and $s\ge 0$, $d=2$ and $s>s_c$, and $d=3$ and $s\ge s_c$ 
under the same conditions. 
(The results in \cite{Hi} for $d\le 3$ and $\theta =0$ 
says that if the nonlinear terms do not have null condition, 
then $s=1$ is optimal regularity to obtain the well-posedness 
by using the iteration argument. ) 
While, it is known that the existence of the blow up solutions 
for the system of nonlinear Schr\"odinger equations. 
Ozawa and Sunagawa (\cite{OS13}) gave the examples of the derivative nonlinearity which causes the small data blow up for a system of Schr\"odinger equations. 
There are also some known results for a system of nonlinear Schr\"odinger equations with no derivative nonlinearity 
(\cite{HLN11}, \cite{HLO11}, \cite{HOT13}). 

To give the main results of the present paper, we first define the 
function space of the Fourier restriction norm.
\begin{defn}
Let $s\in \R$, $b\in \R$, $\sigma \in \R\backslash \{0\}$. \\
(i)\ For $1\le p < \infty$, we define the function space $X^{s,b,p}_{\sigma}$ as the completion of the Schwartz class ${\mathcal S}(\R\times \R^d)$ with the norm
\[
\|u\|_{X^{s,b,p}_{\sigma}}=\left\{\sum_{N\ge 1}N^{2s}
\left(\sum_{L\ge 1}L^{pb}\|Q_{L}^{\sigma}P_{N}u\|_{L^2}^p\right)^\frac{2}{p}\right\}^{\frac{1}{2}},
\]
where $P_N$ and $Q_L^{\sigma}$ will be defined in the last part of this section. \\
(ii)\ We define the function space $X^{s,b,\infty}_{\sigma}$ as the completion of the Schwartz class ${\mathcal S}(\R\times \R^d)$ with the norm
\[
\|u\|_{X^{s,b,\infty}_{\sigma}}=\left\{\sum_{N\ge 1}N^{2s}
\left(\sup_{L\ge 1}L^{b}\|Q_{L}^{\sigma}P_{N}u\|_{L^2}\right)^2\right\}^{\frac{1}{2}}. 
\]
(iii)\ For $1\le p\le \infty$ and $T>0$, 
we define the time localized space $X^{s,b,p}_{\sigma, T}$ as
\[
X^{s,b,p}_{\sigma, T}=\{u|_{[0,T]}|u\in X_{\sigma}^{s,b,p}\}
\]
with the norm
\[
\|u\|_{X^{s,b,p}_{\sigma,T}}=\inf \{\|v\|_{X^{s,b,p}_{\sigma}}|v\in X^{s,b,p}_{\sigma},\ v|_{[0,T]}=u|_{[0,T]}\}.
\]
\end{defn}
The Fourier restriction norm is first introduced by Bourgain in \cite{Bo93}. 
\begin{rem}
The initial datum and the solutions for (\ref{NLS_sys}) 
are $\C^d$-valued function. 
Therefore, $u\in X^{s,b,p}_{\sigma}$ means 
$u^{(j)}\in X^{s,b,p}_{\sigma}$\ ($j=1,\cdots, d)$ and 
$\|u\|_{X^{s,b,p}_{\sigma}}$ means $\sum_{j=1}^d\|u^{(j)}\|_{X^{s,b,p}_{\sigma}}$ 
for $u=(u^{(1)},\cdots,u^{(d)})$
in this paper. 
Similarly, $u_0\in H^s$ means $u_0^{(j)}\in H^s$\ ($j=1,\cdots, d)$ and 
$\|u_0\|_{H^s}$ means $\sum_{j=1}^d\|u_0^{(j)}\|_{H^s}$ 
for $u_0=(u_0^{(1)},\cdots,u_0^{(d)})$. 
\end{rem}
Now, we give the main results in this paper. 
For a Banach space $H$ and $r>0$, we define $B_r(H):=\{ f\in H \,|\, \|f\|_H \le r \}$. 
\begin{thm}\label{wellposed_1}
We assume that $\alpha$, $\beta$, $\gamma \in \R\backslash \{0\}$ satisfy
$\theta <0$ and $\kappa \ne 0$.  \\
{\rm (i)} Let $d=2$ and $s\ge \frac{1}{2}$. Then, 
 {\rm (\ref{NLS_sys})} is locally well-posed in $H^{s}$. 
More precisely, for any $r>0$ and for all initial data $(u_{0}, v_{0}, w_{0})\in B_{r}(H^{s}\times H^{s}\times H^{s})$, there exist $T=T(r)>0$ and a solution
\[
(u,v,w)\in X^{s,\frac{1}{2},1}_{\alpha,T}\times X^{s,\frac{1}{2},1}_{\beta,T}\times X^{s,\frac{1}{2},1}_{\gamma,T}
\]
of the system {\rm (\ref{NLS_sys})} on $[0, T]$. 
Such solution is unique in $B_R(X^{s,\frac{1}{2},1}_{\alpha,T}\times X^{s,\frac{1}{2},1}_{\beta,T}\times X^{s,\frac{1}{2},1}_{\gamma,T})$ for some $R>0$. 
Moreover, the flow map
\[
S_{+}:B_{r}(H^{s}\times H^{s}\times H^{s})\ni (u_{0},v_{0},w_{0})\mapsto (u,v,w)\in X^{s,\frac{1}{2},1}_{\alpha,T}\times X^{s,\frac{1}{2},1}_{\beta,T}\times X^{s,\frac{1}{2},1}_{\gamma,T}
\]
is Lipschitz continuous. \\
{\rm (ii)} Let $d=3$ and $s> \frac{1}{2}$. 
Then, {\rm (\ref{NLS_sys})} is locally well-posed in $H^{s}$. 
\end{thm}
We make a comment on Theorem \ref{wellposed_1}. 
In \cite{Hi}, the first author proved that 
the flow map is not $C^2$ for $s<1/2$ 
under the condition $\theta <0$ and $\kappa \ne 0$. 
Therefore, the above result is optimal as long as we use the iteration argument. 
In \cite{Hi}, we needed the condition $s \geq 1$ to show the key nonlinear estimates for ``resonance'' interactions which is the most difficult interactions to estimate since we cannot recover a derivative loss from modulations. 
To overcome this, we employ a new estimate which was introduced in \cite{BKW}, \cite{BHT10} 
and applied to the Zakharov system in \cite{BHHT09} and \cite{BH11}. 
See Proposition \ref{prop2.7} below.
\kuuhaku \\

\noindent {\bf Notation.} 
We denote the spatial Fourier transform by\ \ $\widehat{\cdot}$\ \ or $\F_{x}$, 
the Fourier transform in time by $\F_{t}$ and the Fourier transform in all variables by\ \ $\widetilde{\cdot}$\ \ or $\F_{tx}$. 
For $\sigma \in \R$, the free evolution $e^{it\sigma \Delta}$ on $L^{2}$ is given as a Fourier multiplier
\[
\F_{x}[e^{it\sigma \Delta}f](\xi )=e^{-it\sigma |\xi |^{2}}\widehat{f}(\xi ). 
\]
We will use $A\lesssim B$ to denote an estimate of the form $A \le CB$ for some constant $C$ and write $A \sim B$ to mean $A \lesssim B$ and $B \lesssim A$. 
We will use the convention that capital letters denote dyadic numbers, e.g. $N=2^{n}$ for $n\in \N_0:=\N\cup\{0\}$ and for a dyadic summation we write
$\sum_{N}a_{N}:=\sum_{n\in \N_0}a_{2^{n}}$ and $\sum_{N\geq M}a_{N}:=\sum_{n\in \N_0, 2^{n}\geq M}a_{2^{n}}$ for brevity. 
Let $\chi \in C^{\infty}_{0}((-2,2))$ be an even, non-negative function such that $\chi (t)=1$ for $|t|\leq 1$. 
We define $\psi (t):=\chi (t)-\chi (2t)$, 
$\psi_1(t):=\chi (t)$, and $\psi_{N}(t):=\psi (N^{-1}t)$ for $N\ge 2$. 
Then, $\sum_{N}\psi_{N}(t)=1$.  
We define frequency and modulation projections
\[
\widehat{P_{N}u}(\xi ):=\psi_{N}(\xi )\widehat{u}(\xi ),\ 
\widetilde{Q_{L}^{\sigma}u}(\tau ,\xi ):=\psi_{L}(\tau +\sigma |\xi|^{2})\widetilde{u}(\tau ,\xi ).
\]
Furthermore, we define $Q_{\geq M}^{\sigma}:=\sum_{L\geq M}Q_{L}^{\sigma}$ and $Q_{<M}:=Id -Q_{\geq M}$. 

The rest of this paper is planned as follows.
In Section 2, we will give the bilinear estimates which will be used to prove the well-posedness.
In Section 3, we will give the proof of the well-posedness. 
\section{Bilinear estimates \label{be_for_2d}}
%
%
In this section, we prove the following bilinear estimate which plays a central role to show Theorem \ref{wellposed_1}.
\begin{prop}\label{key_be}
Let $s\ge \frac{1}{2}$ if $d=2$ and $s>\frac{1}{2}$ if $d=3$.  
Let $\sigma_1$, $\sigma_2$, $\sigma_3\in \R\backslash \{0\}$ 
satisfy $\sigma_1\sigma_2\sigma_3(\frac{1}{\sigma_1}+\frac{1}{\sigma_2}+\frac{1}{\sigma_3})<0$ and $(\sigma_1+\sigma_2)(\sigma_2+\sigma_3)(\sigma_3+\sigma_1)\ne 0$. 
Then there exists $b'\in (0,\frac{1}{2})$ and $C>0$ such that
\begin{align*}
 \|\nabla (u \cdot v)\|_{X^{s,-b',\infty}_{-\sigma_3}}
+ \|(\nabla \cdot u)  v\|_{X^{s,-b',\infty}_{-\sigma_3}}  
+ \|u (\nabla \cdot v ) \|_{X^{s,-b',\infty}_{-\sigma_3}}& \\
\le C\|u\|_{X^{s,b',1}_{\sigma_1}} & \|v\|_{X^{s,b',1}_{\sigma_2}}. 
\end{align*}
\end{prop}
To prove Proposition~\ref{key_be}, we first give the Strichartz estimate. 
\begin{prop}[Strichartz estimate]\label{Stri_est}
Let $\sigma \in \R\backslash \{0\}$ and $(p,q)$ be an admissible pair of exponents for the Schr\"odinger equation, i.e. $p\ge 2$, 
$\frac{2}{p} =d(\frac{1}{2}-\frac{1}{q})$, $(d,p,q)\ne (2,2,\infty)$. Then, we have
\[
\|e^{it\sigma \Delta}\varphi \|_{L_{t}^{p}L_{x}^{q}}\lesssim \|\varphi \|_{L^{2}_{x}}.
\]
for any $\varphi \in L^{2}(\R^{d})$. 
\end{prop}
The Strichartz estimate implies the following. 
(See the proof of Lemma\ 2.3 in \cite{GTV97}.)
\begin{cor}\label{Bo_Stri}
Let $L\in 2^{\N_0}$, $\sigma \in \R\backslash \{0\}$, 
and $(p,q)$ be an admissible pair of exponents for the Schr\"odinger equation. 
Then, we have
\begin{equation}\label{Stri_est_2}
\|Q_{L}^{\sigma}u\|_{L_{t}^{p}L_{x}^{q}}\lesssim L^{\frac{1}{2}}\|Q_{L}^{\sigma}u\|_{L^{2}_{tx}}.
\end{equation}
for any $u \in L^{2}(\R\times \R^{d})$. 
\end{cor}
Next, we give the bilinear Strichartz estimate.
\begin{prop}\label{L2be}
Let $d\in \N$, and $\sigma_{1}$, $\sigma_{2}\in \R \backslash \{0\}$ 
satisfy $\sigma_1+\sigma_2\ne 0$. 
For any dyadic numbers $N_1$, $N_2$, $N_3\in 2^{\N_0}$ 
and $L_1$, $L_2\in 2^{\N_0}$, we have
\begin{equation}\label{L2be_est}
\begin{split}
&\|P_{N_3}(Q_{L_1}^{\sigma_1}P_{N_1}u_{1}\cdot Q_{L_2}^{\sigma_2}P_{N_2}u_{2})\|_{L^{2}_{tx}}\\
&\lesssim N_{\min}^{\frac{d}{2}-1}\left(\frac{N_{\min}}{N_{\max}}\right)^{\frac{1}{2}}L_1^{\frac{1}{2}}L_2^{\frac{1}{2}}
\|Q_{L_1}^{\sigma_1}P_{N_1}u_{1}\|_{L^2_{tx}}\|Q_{L_2}^{\sigma_2}P_{N_2}u_{2}\|_{L^2_{tx}}, 
\end{split}
\end{equation}
where $N_{\min}=\displaystyle \min_{1\le i\le 3}N_i$, 
$N_{\max}=\displaystyle \max_{1\le i\le 3}N_i$.
\end{prop}
\begin{proof} 
By the symmetry, we ca assume $N_1\ge N_2$. 
For the case $N_3\sim N_1\gg N_2$, the proof is same as 
the proof of Lemma\ {\rm 3.1} in \cite{Hi}. 
For the case $N_1\sim N_2\sim N_3$, we can obtain (\ref{L2be_est}) 
by using the H\"older inequality, the Bernstein inequality, and the Strichartz estimate
(\ref{Stri_est}) with $(p,q)=(4,\frac{2d}{d-1})$. 
Now, we consider the case $N_1\sim N_2\gg N_3$. 
Since $\sigma_1+\sigma_2\ne 0$, 
if $|\xi_1|\sim |\xi_2|\sim N_1$ and $|\xi_1+\xi_2|\sim N_3$ hold, then 
$|\sigma_1\xi_1-\sigma_2\xi_2|\sim N_1$. 
We assume $|\sigma_1\xi_1^{(1)}-\sigma_2\xi_2^{(1)}|\sim N_1$ 
for $\xi_i=(\xi_i^{(1)},\cdots \xi_i^{(d)})$ $(i=1,2)$. 
We divide $\R^2$ into cubes $\{B_k\}_k$ with width $2N_3$,
and decompose
\[
u_1=\sum_{k}P_{B_k}u_1, 
\]
where $\widehat{P_{B_k}u_1}=\ee_{B_k}\widehat{u_1}$. 
Let $\zeta_k\in B_k$ be a center of $B_k$. 
If $|\xi_1-\zeta_k|\le 2N_3$ and $|\xi_1+\xi_2|\le 2N_3$ hold, 
then it holds that $|\xi_2+\zeta_k|\le 4N_3$. 
Therefore, if we put $C_k$ as the cube with center $-\zeta_k$ and width $4N_3$,  
then we have
\[
\begin{split}
&\|P_{N_3}(Q_{L_1}^{\sigma_1}P_{N_1}u_{1}\cdot Q_{L_2}^{\sigma_2}P_{N_2}u_{2})\|_{L^{2}_{tx}}
\lesssim \sum_{k}\|u_{1,N_1,L_1,k}\cdot u_{2,N_2,L_2,k}\|_{L^2_{tx}}, 
\end{split}
\]
where
\[
u_{1,N_1,L_1,k}:=P_{B_k}Q_{L_1}^{\sigma_1}P_{N_1}u_{1},\ \ 
u_{2,N_2,L_2,k}:=P_{C_k}Q_{L_2}^{\sigma_2}P_{N_2}u_{2}.
\] 
We put $f_{i,k}=\F[u_{i,N_i,L_i,k}]$ $(i=1,2)$. 
By the duality argument and
\[
\begin{split}
\sum_k \|f_{1,k}\|_{L^2_{\tau \xi}}\|f_{2,k}\|_{L^2_{\tau \xi}}
&\le \left(\sum_k\|f_{1,k}\|_{L^2_{\tau \xi}}^2\right)^{\frac{1}{2}}
\left(\sum_k\|f_{2,k}\|_{L^2_{\tau \xi}}^2\right)^{\frac{1}{2}}\\
&\lesssim \|Q_{L_1}^{\sigma_1}P_{N_1}u_{1}\|_{L^2_{tx}}
\|Q_{L_2}^{\sigma_2}P_{N_2}u_{2}\|_{L^2_{tx}},
\end{split}
\]
it suffice to show that
\begin{equation}\label{BSE_pf_1}
\begin{split}
&\left|\int_{\Omega_k}f_{1,k}(\tau_1,\xi_1)f_{2,k}(\tau_2,\xi_2)f(\tau_1+\tau_2,\xi_1+\xi_2)d\tau_1d\tau_2d\xi_1d\xi_2\right|\\
&\lesssim N_{3}^{\frac{d}{2}-1}\left(\frac{N_{3}}{N_{1}}\right)^{\frac{1}{2}}L_1^{\frac{1}{2}}L_2^{\frac{1}{2}}
\|f_{1,k}\|_{L^2_{\tau \xi}}\|f_{2,k}\|_{L^2_{\tau \xi}}\|f\|_{L^2_{\tau \xi}}
\end{split}
\end{equation}
for any $f\in L^2(\R\times \R^d)$, where
\[
\Omega_k=\{(\tau_1,\tau_2,\xi_1,\xi_2)|\ |\xi_i|\sim N_i,\ |\tau_i+\sigma_i|\xi|^2|\sim L_i,\ (i=1,2),\ |\xi_1-\zeta_k|\lesssim N_3,\ |\xi_2+\zeta_k|\lesssim N_3\}.
\]
By the Cauchy-Schwartz inequality, we have
\begin{equation}\label{BSE_pf_2}
\begin{split}
&\left|\int_{\Omega_k}f_{1,k}(\tau_1,\xi_1)f_{2,k}(\tau_2,\xi_2)f(\tau_1+\tau_2,\xi_1+\xi_2)d\tau_1d\tau_2d\xi_1d\xi_2\right|\\
&\lesssim \|f_{1,k}\|_{L^2_{\tau\xi}}\|f_{2,k}\|_{L^2_{\tau\xi}}
\left(\int_{\Omega_k}|f(\tau_1+\tau_2,\xi_1+\xi_2)|^2d\tau_1d\tau_2d\xi_1d\xi_2\right)^{\frac{1}{2}}.
\end{split}
\end{equation}
By applying the variable transform $(\tau_1,\tau_2)\mapsto (\theta_1,\theta_2)$ and $(\xi_1,\xi_2)\mapsto (\mu,\nu, \eta)$ as 
\[
\begin{split}
&\theta_i=\tau_i+\sigma_i|\xi_i|^2\ \  (i=1,2),\\
&\mu =\theta_1+\theta_2-\sigma_1|\xi_1|^2-\sigma_2|\xi_2|^2,\ \nu=\xi_1+\xi_2,\ 
\eta=(\xi_2^{(2)},\cdots, \xi_2^{(d)}),
\end{split}
\]
we have
\[
\begin{split}
&\int_{\Omega_k}|f(\tau_1+\tau_2,\xi_1+\xi_2)|^2d\tau_1d\tau_2d\xi_1d\xi_2\\
&\lesssim \int_{\substack{|\theta_1|\sim L_1\\ |\theta_2|\sim L_2}}\left(\int_{|\eta +\overline{\zeta}_k|\lesssim N_3}|f(\mu,\nu)|^2
J(\xi_1,\xi_2)^{-1}d\mu d\nu d\eta \right)d\theta_1d\theta_2 ,
\end{split}
\]
where $\overline{\zeta}_k=(\zeta_k^{(2)},\cdots, \zeta_k^{(d)})$ and 
\[
J(\xi_1,\xi_2)
=\left|{\rm det}\frac{\partial (\mu ,\nu,\eta )}{\partial (\xi_1,\xi_2)}\right|
=2|\sigma_1\xi_1^{(1)}-\sigma_2\xi_2^{(1)}|\sim N_{1}.
\]
Therefore, we obtain
\begin{equation}\label{BSE_pf_3}
\int_{\Omega_k}|f(\tau_1+\tau_2,\xi_1+\xi_2)|^2d\tau_1d\tau_2d\xi_1d\xi_2
\lesssim N_3^{d-1}N_1^{-1}L_1L_2\|f\|_{L^2_{\tau\xi\eta}}.
\end{equation}
As a result, we get (\ref{BSE_pf_1}) from (\ref{BSE_pf_2}) and (\ref{BSE_pf_3}).
\end{proof}
\begin{cor}\label{L2be_2}
Let $d\in \N$, $b'\in (\frac{1}{4},\frac{1}{2})$, 
and 
$\sigma_{1}$, $\sigma_{2}\in \R \backslash \{0\}$ satisfy $\sigma_1+\sigma_2\ne 0$, 
We put $\delta =\frac{1}{2}-b'$. 
For any dyadic numbers $N_1$, $N_2$, $N_3\in 2^{\N_0}$ 
and $L_1$, $L_2\in 2^{\N_0}$, we have
\begin{equation}\label{L2be_est_2}
\begin{split}
&\|P_{N_3}(Q_{L_1}^{\sigma_1}P_{N_1}u_{1}\cdot Q_{L_2}^{\sigma_2}P_{N_2}u_{2})\|_{L^{2}_{tx}}\\
&\lesssim 
N_{\min}^{\frac{d}{2}-1+4 \delta}
\left(\frac{N_{\min}}{N_{\max}}\right)^{\frac{1}{2}- 2\delta}L_1^{b'}L_2^{b'}
\|Q_{L_1}^{\sigma_1}P_{N_1}u_{1}\|_{L^2_{tx}}\|Q_{L_2}^{\sigma_2}P_{N_2}u_{2}\|_{L^2_{tx}}.  
\end{split}
\end{equation}
\end{cor}
\begin{proof}
The desired estimate is obtained by the interpolation between (\ref{L2be_est}) and the following bilinear estimate:
\begin{equation}\label{L2be_est_3}
\begin{split}
&\|P_{N_3}(Q_{L_1}^{\sigma_1}P_{N_1}u_{1}\cdot Q_{L_2}^{\sigma_2}P_{N_2}u_{2})\|_{L^{2}}\lesssim N_{\min}^{\frac{d}{2}}L_1^{\frac{1}{4}}L_2^{\frac{1}{4}}
\|Q_{L_1}^{\sigma_1}P_{N_1}u_{1}\|_{L^2}
\|Q_{L_2}^{\sigma_1}P_{N_2}u_{2}\|_{L^2}.
\end{split}
\end{equation}
Therefore, we only need to show (\ref{L2be_est_3}). 
By the same argument as in the proof of Proposition \ref{L2be}, we may assume that 
$\supp \F_{tx} u_1$ and $\supp \F_{tx} u_2$ are contained in the cubes $B_k$ and $B_{j(k)}$, respectively. 
Here the cubes $\{ B_k \}_k$ denote the decomposition of $\R^2$ with width $2 N_{\min}$. 
By the H\"older inequality, the Bernstein inequality, 
and the Strichartz estimate (\ref{Stri_est_2}) with 
$(p,q)=(\infty ,2)$, we have
\begin{equation*}
\begin{split}
&\|P_{N_3}(Q_{L_1}^{\sigma_1}P_{N_1}u_{1}\cdot Q_{L_2}^{\sigma_2}P_{N_2}u_{2})\|_{L^{2}}\\
&\lesssim \|Q_{L_1}^{\sigma_1}P_{N_1}u_{1}\|_{L^2}^{\frac{1}{2}}
\|Q_{L_1}^{\sigma_1}P_{N_1}u_{1}\|_{L^{\infty}}^{\frac{1}{2}}
\|Q_{L_2}^{\sigma_1}P_{N_2}u_{2}\|_{L^2}^{\frac{1}{2}}
\|Q_{L_2}^{\sigma_1}P_{N_2}u_{2}\|_{L^{\infty}}^{\frac{1}{2}}\\
&\lesssim N_{\min}^{\frac{d}{2}} L_1^{\frac{1}{4}}L_2^{\frac{1}{4}}
\|Q_{L_1}^{\sigma_1}P_{N_1}u_{1}\|_{L^2}
\|Q_{L_2}^{\sigma_1}P_{N_2}u_{2}\|_{L^2}
\end{split}
\end{equation*}
which completes the proof of \eqref{L2be_est_3}.
\end{proof}
%
%
%
%
\subsection{The estimates for low modulation, $2$D}
In this subsection, we assume that $L_{\textnormal{max}} \ll N_{\max}^2$ and $d=2$. 
In this case, we cannot recover a derivative loss by using $L_{\textnormal{max}} \gtrsim N_{\max}^2$. 
Therefore, the strategy for the case $L_{\textnormal{max}} \gtrsim N_{\max}^2$ is no longer available. 
However, thanks to $\kappa \not= 0$, the following relation holds. 
\begin{lemm}\label{modul_est_2}
Let $s \in \N$. We assume that $\sigma_1$, $\sigma_2$, $\sigma_3 \in \R \setminus \{0 \}$ satisfy 
$(\sigma_1 + \sigma_2) (\sigma_2 + \sigma_3) (\sigma_{3}+\sigma_{1})\neq 0$ and $(\tau_{1},\xi_{1})$, $(\tau_{2}, \xi_{2})$, $(\tau_{3}, \xi_{3})\in \R\times \R^{d}$ satisfy $\tau_{1}+\tau_{2}+\tau_{3}=0$, $\xi_{1}+\xi_{2}+\xi_{3}=0$. 
If $\displaystyle \max_{1\leq j\leq 3}|\tau_{j}+\sigma_{j}|\xi_{j}|^{2}|
\ll \max_{1\leq j\leq 3}|\xi_{j}|^{2}$ then we have
\begin{equation*}
|\xi_1| \sim |\xi_2| \sim |\xi_3|.
\end{equation*}
\end{lemm}
Since the above lemma is the contrapositive of the following lemma which was utilized in \cite{Hi}, we omit the proof. 
\begin{lemm}[Lemma\ 4.1\ in \cite{Hi}]\label{modul_est}
Let $d\in \N$. We assume that $\sigma_{1}$, $\sigma_{2}$, $\sigma_{3} \in \R \backslash \{0\}$ satisfy $(\sigma_{1}+\sigma_{2})(\sigma_{2}+\sigma_{3})(\sigma_{3}+\sigma_{1})\neq 0$ and $(\tau_{1},\xi_{1})$, $(\tau_{2}, \xi_{2})$, $(\tau_{3}, \xi_{3})\in \R\times \R^{d}$ satisfy $\tau_{1}+\tau_{2}+\tau_{3}=0$, $\xi_{1}+\xi_{2}+\xi_{3}=0$.  
If there exist $1\leq i,j\leq 3$ such that $|\xi_{i}|\ll |\xi_{j}|$, then we have
\begin{equation}\label{modulation_est}
\max_{1\leq j\leq 3}|\tau_{j}+\sigma_{j}|\xi_{j}|^{2}|
\gtrsim \max_{1\leq j\leq 3}|\xi_{j}|^{2}. 
\end{equation}
\end{lemm}
Lemma \ref{modul_est_2} suggests that if  $\displaystyle \max_{1\leq j\leq 3}|\tau_{j}+\sigma_{j}|\xi_{j}|^{2}|
\ll \max_{1\leq j\leq 3}|\xi_{j}|^{2}$ then we can assume
\begin{equation}
\max_{1 \leq j\leq 3} |\tau_{j}+\sigma_{j}|\xi_{j}|^{2}| \ll  
\min_{1\leq j\leq 3} |\xi_j|^2.
\end{equation}
We first introduce the angular frequency localization operators which were utilized in \cite{BHHT09}.
\begin{defn}[\cite{BHHT09}]
We define the angular decomposition of $\R^3$ in frequency.
We define a partition of unity in $\R$,
\begin{equation*}
1 = \sum_{j \in \Z} \omega_j, \qquad \omega_j (s) = \psi(s-j) \left( \sum_{k \in \Z} \psi (s-k) \right)^{-1}. 
\end{equation*}
For a dyadic number $A \geq 64$, we also define a partition of unity on the unit circle,
\begin{equation*}
1 = \sum_{j =0}^{A-1} \omega_j^A, \qquad \omega_j^A (\theta) = 
\omega_j \left( \frac{A\theta}{\pi} \right) + \omega_{j-A} \left( \frac{A\theta}{\pi} \right).
\end{equation*}
We observe that $\omega_j^A$ is supported in 
\begin{equation*}
\Theta_j^A = \left[\frac{\pi}{A} \, (j-2), \ \frac{\pi}{A} \, (j+2) \right] 
\cup \left[-\pi + \frac{\pi}{A} \, (j-2), \ - \pi +\frac{\pi}{A} \, (j+2) \right].
\end{equation*}
We now define the angular frequency localization operators $R_j^A$,
\begin{equation*}
\F_x (R_j^A f)(\xi) = \omega_j^A(\theta) \F_x f(\xi), \qquad \textnormal{where} \ \xi = |\xi| 
(\cos \theta, \sin \theta).
\end{equation*}
For any function $u  : \, \R \, \times \, \R^2 \, \to \C$, $(t,x) \mapsto u(t,x)$ we set 
$(R_j^A u ) (t, x) = (R_j^Au( t, \cdot)) (x)$. These operators localize function in frequency to the sets
\begin{equation*}
{\mathfrak{D}}_j^A = \{ (\tau, |\xi| \cos \theta, |\xi| \sin \theta) \in \R \times \R^2 
\, | \, \theta \in \Theta_j^A  \} .
\end{equation*}
Immediately, we can see
\begin{equation*}
u = \sum_{j=0}^{A-1} R_j^A u.
\end{equation*}
\end{defn}
Now we introduce the necessary bilinear estimates for $2$D.
\begin{thm}\label{thm-0.3}
Let $\displaystyle L_{\max} := \max_{1\leq j\leq 3} (L_1, L_2, L_3) \ll |\theta| N_{\min}^2$, 
$A \geq 64$ and $|j_1 - j_2| \lesssim 1$. 
Then the following estimates holds:
\begin{align}
\begin{split}
\|Q_{L_3}^{-\sigma_3} P_{N_3}(R_{j_1}^A Q_{L_1}^{\sigma_1}P_{N_1}u_{1}\cdot 
R_{j_2}^A Q_{L_2}^{\sigma_2}P_{N_2}u_{2})\|_{L^{2}_{tx}} & \\
\lesssim A^{-\frac{1}{2}} L_1^{\frac{1}{2}}L_2^{\frac{1}{2}} 
\|R_{j_1}^A Q_{L_1}^{\sigma_1}P_{N_1}u_{1}\|_{L^2_{tx}} &  \|R_{j_2}^A Q_{L_2}^{\sigma_2}P_{N_2}u_{2}\|_{L^2_{tx}},
\end{split}\label{bilinear-12}\\
\begin{split}
\|R_{j_1}^A Q_{L_1}^{-\sigma_1} P_{N_1}(R_{j_2}^A Q_{L_2}^{\sigma_2}P_{N_2}u_{2}\cdot 
 Q_{L_3}^{\sigma_3}P_{N_3}u_{3})\|_{L^{2}_{tx}} & \\
\lesssim A^{-\frac{1}{2}} L_2^{\frac{1}{2}}L_3^{\frac{1}{2}} 
\|R_{j_2}^A Q_{L_2}^{\sigma_2}P_{N_2}u_{2}\|_{L^2_{tx}} &  \|Q_{L_3}^{\sigma_3}P_{N_3}u_{3}\|_{L^2_{tx}},
\end{split}\label{bilinear-23}\\
\begin{split}
\|R_{j_2}^A Q_{L_2}^{-\sigma_2} P_{N_2}( Q_{L_3}^{\sigma_3}P_{N_3}u_{3}\cdot 
 R_{j_1}^A Q_{L_1}^{\sigma_1}P_{N_1}u_{1})\|_{L^{2}_{tx}} & \\
\lesssim A^{-\frac{1}{2}} L_3^{\frac{1}{2}}L_1^{\frac{1}{2}} \|Q_{L_3}^{\sigma_3}P_{N_3}u_{3}\|_{L^2_{tx}} & 
\|R_{j_1}^A Q_{L_1}^{\sigma_1}P_{N_1}u_{1}\|_{L^2_{tx}}.
\end{split}\label{bilinear-31}
\end{align}
\end{thm}
\begin{proof}
If $A \sim 1$, Proposition \ref{L2be} implies \eqref{bilinear-12}-\eqref{bilinear-31}. 
Then we assume that $A$ is sufficiently large. 
Also, we can assume $N_1 \sim N_2 \sim N_3$ from Lemma \ref{modul_est_2}. 
Thus it suffices to show \eqref{bilinear-12}.
Indeed, thanks to $N_1 \sim N_2 \sim N_3$, we may replace $u_3$ in \eqref{bilinear-23} and \eqref{bilinear-31} with 
$R_j^A u_3$ where $j$ satisfies $|j-j_1|$, $|j-j_2| \lesssim 1$. 
Therefore, here we prove only \eqref{bilinear-12}. 

Let $f_{i}=\F_{tx} [R_{j_i}^A Q_{L_i}^{\sigma_i}P_{N_i}u_{i}]$ $(i=1,2)$. 
By Plancherel's theorem, we may rewrite \eqref{bilinear-12} as
\begin{equation}
\begin{split}
 \left\| \psi_{L_3} (\tau - \sigma_3 |\xi|^2) \psi_{N_3}(\xi) \int 
f_1 (\tau_1,\xi_1) f_2 (\tau-\tau_1, \xi-\xi_1)d\tau_1d\xi_1 \right\|_{L_{\tau \xi}^2} & \\ 
\lesssim 
A^{-\frac{1}{2}} L_1^{\frac{1}{2}}L_2^{\frac{1}{2}} 
\|f_1\|_{L^2_{\tau \xi}} & \|f_2 \|_{L^2_{\tau \xi}}.
\end{split}\label{bilinear-15}
\end{equation}
Let $\psi_{N_3, L_3}^{\sigma_3}(\ta,\xi) := \psi_{L_3} (\tau - \sigma_3 |\xi|^2) \psi_{N_3}(\xi)$.
We calculate as
\begin{align*}
&  \left\|\psi_{N_3, L_3}^{\sigma_3}(\ta,\xi) \int 
f_1 (\tau_1, \xi_1) f_2 (\tau-\tau_1, \xi-\xi_1)d\tau_1d\xi_1 \right\|_{L_{\tau \xi}^2}\\
\lesssim & \left\| \psi_{N_3, L_3}^{\sigma_3}(\ta,\xi)  \left( \int |f_1|^2 (\tau_1, \xi_1) 
| f_2|^2 (\tau - \tau_1, \xi-\xi_1) d\tau_1d\xi_1 \right)^{1/2} (E(\tau, \xi))^{1/2}  \right\|_{L_{\tau \xi}^2}\\
\lesssim & \sup_{(\tau,\xi) \in \supp \psi_{N_3, L_3}^{\sigma_3}} |E(\tau, \xi)|^{1/2} \| |f_1|^2 * |f_2|^2 \|_{L_{\tau \xi}^1}^{1/2}\\
\lesssim & \sup_{(\tau,\xi) \in \supp \psi_{N_3, L_3}^{\sigma_3}} |E(\tau, \xi)|^{1/2} \|f_1 \|_{ L_{\tau \xi}^2} \| f_2 \|_{ L_{\tau \xi}^2}.
\end{align*}
Then it suffices to prove
\begin{equation}
 \sup_{(\tau,\xi) \in \supp \psi_{N_3, L_3}^{\sigma_3}} |E(\ta, \xi) | \lesssim A^{-1} L_1 L_2,\label{b0.1}
\end{equation}
where
\begin{equation*} E(\ta, \xi) 
 := \left\{ (\ta_1, \xi_1)  \in  {\mathfrak{D}}_{j_1}^{A} \ \left| \ 
\begin{aligned} & \LR{\ta_1 + \sigma_1 |\xi_1|^2} \sim L_1, \ \LR{\ta - \ta_1 + \sigma_2 |\xi - \xi_1|^2} 
\sim L_2, \\ 
& (\ta-\ta_1, \xi-\xi_1) \in {\mathfrak{D}}_{j_2}^{A}.
 \end{aligned} \right.
\right\}
\end{equation*}
with $ |j_1 - j_2| \lesssim 1$. 
From $\LR{\ta_1 + \sigma_1 |\xi_1|^2} \sim L_1$ and $\LR{\ta - \ta_1 + \sigma_2 |\xi - \xi_1|^2} 
\sim L_2$, for fixed $\xi_1$, 
\begin{equation}
| \{ \ta_1 \ | \ (\ta_1, \xi_1) \in E(\ta, \xi) \} | \lesssim 
\min (L_1, L_2).\label{ele1-lem-a-bilinear}
\end{equation}
Let $\theta_1$ be defined as $\xi_1 := (|\xi_1| \cos \theta_1, |\xi_1| \sin \theta_1)$. 
It follows from
\begin{align*}
& (\tau_1 + \sigma_1 |\xi_1|^2 ) + (\tau - \tau_1 + \sigma_2 |\xi - \xi_1|^2)\\
= & (\tau + \sigma_1 |\xi_1|^2 + \sigma_2 |\xi - \xi_1|^2) \\
= & \tau -\sigma_3 |\xi|^2 + \frac{ 
((\sigma_1 + \sigma_2) |\xi_1| - \sigma_2 \cos{\angle(\xi, \xi_1)} |\xi|)^2 - 
(|\theta| - \sigma_2^2 \sin^2{\angle(\xi, \xi_1)}) |\xi|^2 }{\sigma_1 + \sigma_2} 
\end{align*}
that
\begin{equation}
\begin{split}
& \bigl( (\sigma_1 + \sigma_2) |\xi_1| - \sigma_2 \cos{\angle(\xi, \xi_1)} |\xi| \bigr)^2 \\
= &  -(\sigma_1 + \sigma_2)(\ta -\sigma_3 |\xi|^2) +(|\theta| - \sigma_2^2 \sin^2{\angle(\xi, \xi_1)}) |\xi|^2 
+ \mathcal{O}(\max (L_1, L_2)).
\end{split}\label{modu-1-8}
\end{equation}
Since $A$ is sufficiently large, $\sin{\angle(\xi, \xi_1)} \, (\sim A^{-1})$ is sufficiently small, so that 
we assume $|\theta| - \sigma_2^2 \sin^2{\angle(\xi, \xi_1)} > |\theta|/2$. 
Therefore, for fixed $\theta_1$, \eqref{modu-1-8} tells that $|\xi_1|$ is confined to a set of measure at most 
$\mathcal{O}(\max (L_1, L_2) / N_1)$. 
From $(\ta_1, \xi_1)  \in  {\mathfrak{D}}_{j_1}^{A}$, $\theta_1$ is confined to a set of measure $\sim A^{-1}$.
We observe
\begin{align*}
&  | \{ \xi_1 \ | \ (\ta_1, \xi_1) \in E(\ta, \xi) \} |\\
= &  \int_{\theta_1} \int_{|\xi_1|} {\mathbf 1}_{E(\ta,\xi)} (|\xi_1|, \theta_1) |\xi_1| d|\xi_1| d \theta_1 \\
\lesssim & A^{-1}\max (L_1, L_2).
\end{align*}
Combining \eqref{ele1-lem-a-bilinear}, this completes the proof of \eqref{b0.1}.
\end{proof}
\begin{prop}\label{thm2.6}
Let $L_{\textnormal{max}} \ll |\theta| 
N_{\min}^2$ 
and $ 64 \leq A \leq N_{\textnormal{max}}$, \ $16 \leq |j_1 - j_2 |\leq 32$. 
Then the following estimate holds:
\begin{equation}
\begin{split}
\|Q_{L_3}^{-\sigma_3} P_{N_3}(R_{j_1}^A Q_{L_1}^{\sigma_1}P_{N_1}u_{1}\cdot 
R_{j_2}^A Q_{L_2}^{\sigma_2}P_{N_2}u_{2})\|_{L^{2}_{tx}} & \\
\lesssim A^{\frac{1}{2}} N_1^{-1} L_1^{\frac{1}{2}}L_2^{\frac{1}{2}} L_3^{\frac{1}{2}} 
\|R_{j_1}^A Q_{L_1}^{\sigma_1}P_{N_1}u_{1}\|_{L^2_{tx}} &  \|R_{j_2}^A Q_{L_2}^{\sigma_2}P_{N_2}u_{2}\|_{L^2_{tx}}.
\end{split}\label{0609}
\end{equation}
\end{prop}
For the proof of the above proposition, we introduce the important estimate. See \cite{BH11} for more general case.
\begin{prop}[\cite{BHT10} Corollary 1.5] \label{prop2.7}
Assume that the surface $S_i$ $(i=1,2,3)$ 
is an open and bounded subset of $S_i^*$ which 
satisfies the following conditions \textnormal{(Assumption 1.1 in \cite{BHT10})}.

\textnormal{(i)} $S_i^*$ is defined as
\begin{equation*}
S_i^* = \{ {\lambda_i} \in U_i \ | \ \Phi_i({\lambda_i}) = 0 , \nabla \Phi_i \not= 0, \Phi_i \in C^{1,1} (U_i) \}, 
\end{equation*}
for a convex $U_i \subset \R^3$ such that \textnormal{dist}$(S_i, U_i^c) \geq$ \textnormal{diam}$(S_i)$;

\textnormal{(ii)} the unit normal vector field $\mathfrak{n}_i$ on $S_i^*$ satisfies the H\"{o}lder condition
\begin{equation*}
\sup_{\lambda, \lambda' \in S_i^*} \frac{|\mathfrak{n}_i(\lambda) - 
\mathfrak{n}_i(\lambda')|}{|\lambda - \lambda'|}
+ \frac{|\mathfrak{n}_i(\lambda) ({\lambda} - {\lambda}')|}{|{\lambda} - {\lambda}'|^2} \lesssim 1;
\end{equation*}

\textnormal{(iii)} there exists $d >0$ such that the matrix ${N}({\lambda_1}, {\lambda_2}, {\lambda_3}) = ({\mathfrak{n}_1} 
({\lambda_1}), {\mathfrak{n}_2}({\lambda_2}), {\mathfrak{n}_3}({\lambda_3}))$ 
satisfies the transversality condition
\begin{equation*}
d \leq \textnormal{det} {N}({\lambda_1}, {\lambda_2}, {\lambda_3})  \leq 1
\end{equation*}
for all $({\lambda_1}, {\lambda_2}, {\lambda_3}) \in {S_1^*} \cross {S_2^*} \cross {S_3^*}$.

We also assume \textnormal{diam}$({S_i}) \lesssim d$. 
Then for functions $f \in L^2 (S_1)$ and $g \in L^2 (S_2)$, the restriction of the convolution $f*g$ to 
$S_3$ is a well-defined $L^2(S_3)$-function which satisfies 
\begin{equation*}
\| f *g \|_{L^2(S_3)} \lesssim \frac{1}{\sqrt{d}} \| f \|_{L^2(S_1)} \| g\|_{L^2(S_2)}.
\end{equation*}
\end{prop}
\begin{rem}
(1) If $S_1$, $S_2$, $S_3$ are given coordinate hyperplanes in $\R^3$;
\begin{align*}
S_1 = \{ (x_1, x_2, x_3) \in \R^3 \, | \,  x_1 =0\},\\
S_2 = \{ (x_1, x_2, x_3) \in \R^3 \, | \,  x_2 =0\},\\
S_3 = \{ (x_1, x_2, x_3) \in \R^3 \, | \,  x_3 =0\},
\end{align*}
then the inequality $\| f *g \|_{L^2(S_3)} \lesssim \| f \|_{L^2(S_1)} \| g\|_{L^2(S_2)}$ is known as the classical Loomis-Whitney inequality in $\R^3$ which was introduced in \cite{LW}. 
Thus, we would say that Proposition \ref{prop2.7} is the generalization of the Loomis-Whitney inequality. \\
(2) As was mentioned in \cite{BHT10}, the condition of $S_i^*$ in \textnormal{(i)} is used only to ensure the existence of a global representation of 
$S_i$ as a graph. In the proof of Proposition \ref{thm2.6}, the implicit function theorem and the other conditions may show the existence 
of such a graph. Thus we will not treat the condition \textnormal{(i)} in the proof of Proposition \ref{thm2.6}.
\end{rem}

\begin{proof}[Proof of Proposition \ref{thm2.6}]
We divide the proof into the following two cases:
\begin{equation*}
\textnormal{(I)} \quad  L_{\textnormal{max}} \geq A^{-1} N_1^2, \qquad 
\textnormal{(I \hspace{-0.15cm}I)} \quad L_{\textnormal{max}} \leq  A^{-1} N_1^2.
\end{equation*}
We first consider the case $\textnormal{(I)}$. 
We subdivide the proof further.
\begin{equation*}
\textnormal{(Ia)} \quad  L_3 \geq A^{-1} N_1^2, \qquad 
\textnormal{(Ib)} \quad L_1 \geq  A^{-1} N_1^2, \qquad
\textnormal{(Ic)} \quad L_2 \geq  A^{-1} N_1^2.
\end{equation*}
For the case $\textnormal{(Ia)}$, we use the estimate \eqref{bilinear-12} in Theorem \ref{thm-0.3}.
\begin{align*}
& \|Q_{L_3}^{-\sigma_3} P_{N_3}(R_{j_1}^A Q_{L_1}^{\sigma_1}P_{N_1}u_{1}\cdot 
R_{j_2}^A Q_{L_2}^{\sigma_2}P_{N_2}u_{2})\|_{L^{2}_{tx}}  \\
\lesssim & A^{- \frac{1}{2}}  L_1^{\frac{1}{2}}L_2^{\frac{1}{2}}  
\|R_{j_1}^A Q_{L_1}^{\sigma_1}P_{N_1}u_{1}\|_{L^2_{tx}}   \|R_{j_2}^A Q_{L_2}^{\sigma_2}P_{N_2}u_{2}\|_{L^2_{tx}}\\
\lesssim &  N_1^{-1} L_1^{\frac{1}{2}}L_2^{\frac{1}{2}} L_3^{\frac{1}{2}}  
\|R_{j_1}^A Q_{L_1}^{\sigma_1}P_{N_1}u_{1}\|_{L^2_{tx}}   \|R_{j_2}^A Q_{L_2}^{\sigma_2}P_{N_2}u_{2}\|_{L^2_{tx}}.
\end{align*}
For $\textnormal{(Ib)}$, by the dual estimate, H\"{o}lder inequality and \eqref{bilinear-23}, we have
\begin{align*}
& \|Q_{L_3}^{-\sigma_3} P_{N_3}(R_{j_1}^A Q_{L_1}^{\sigma_1}P_{N_1}u_{1}\cdot 
R_{j_2}^A Q_{L_2}^{\sigma_2}P_{N_2}u_{2})\|_{L^{2}_{tx}}\\
\sim & \sup_{\|u_3 \|_{L^2} =1} \left| \int (R_{j_1}^A Q_{L_1}^{\sigma_1}P_{N_1}u_{1}) \,  
(R_{j_2}^A Q_{L_2}^{\sigma_2}P_{N_2}u_{2}) \, (Q_{L_3}^{\sigma_3} P_{N_3} u_3 ) \, dxdt \right|\\
\lesssim & \| R_{j_1}^A Q_{L_1}^{\sigma_1}P_{N_1}u_{1} \|_{L^2} 
\sup_{\|u_3 \|_{L^2} =1} \|R_{j_1}^A Q_{L_1}^{-\sigma_1} P_{N_1}(R_{j_2}^A Q_{L_2}^{\sigma_2}P_{N_2}u_{2}\cdot 
 Q_{L_3}^{\sigma_3}P_{N_3}u_{3})\|_{L^{2}_{tx}}  \\
\lesssim & N^{-1}_1 L_1^{\frac{1}{2}} L_2^{\frac{1}{2}}L_3^{\frac{1}{2}} 
\| R_{j_1}^A Q_{L_1}^{\sigma_1}P_{N_1}u_{1} \|_{L^2} 
\|R_{j_2}^A Q_{L_2}^{\sigma_2}P_{N_2}u_{2}\|_{L^2_{tx}}.
\end{align*}
The case $\textnormal{(Ic)}$ can be treated similarly. 

For $\textnormal{(I \hspace{-0.15cm}I)}$, by Plancherel's theorem and the dual estimate, \eqref{0609} is verified by the following estimate:
\begin{equation}
\begin{split}
 \left| \int  f_1 (\ta_1, \xi_1) f_2 (\ta_2, \xi_2) f_3 (\tau_1+\tau_2, \xi_1+ \xi_2) d\ta_1 d\ta_2 d\xi_1 d \xi_2 \right| & \\ \lesssim 
A^{\frac{1}{2}} N_1^{-1} ( L_1 L_2 L_3)^\frac{1}{2}  
\|f_1\|_{L^2_{\tau \xi}} & \|f_2 \|_{L^2_{\tau \xi}}  \|f_3 \|_{L^2_{\tau \xi}}\label{0609-2}
\end{split}
\end{equation}
where $f_{i}=\F_{tx} [R_{j_i}^A Q_{L_i}^{\sigma_i}P_{N_i}u_{i}]$ $(i=1,2)$ and $f_{3}=\F_{tx} [Q_{L_3}^{- \sigma_3}P_{N_3}u_{3}]$. To show \eqref{0609-2}, we first decompose $f_1$ by thickened circular 
localization characteristic functions $\left\{ {\mathbf 1}_{\mathbb{S}_{\delta}^{{N_1} + k \delta}} \right\}_{k=0}^{\left[\frac{N_1}{\delta}\right]+1}$ 
where $[s]$ denotes the maximal integer which is not greater than $s \in \R$ and $\mathbb{S}_\delta^{\xi^0} = \{ (\ta, \xi ) \in \R \cross \R^2 \ | \ \xi^0 \leq \LR{\xi}\leq \xi^0 + \delta \}$ with 
$\delta = A^{-1} N_1$ as follows:
\begin{equation*}
f_1 = \sum_{k= 0}^{\left[\frac{N_1}{\delta}\right]+1} 
{\mathbf 1}_{\mathbb{S}_{\delta}^{N_1 + k \delta} } f_1.
\end{equation*}
Thanks to $L_{\textnormal{max}} \leq  A^{-1} N_1^2$, for each 
$f_{1, k} := {\mathbf 1}_{\mathbb{S}_{\delta}^{N_1 + k \delta} } f_1$ with fixed $k \in [0, \left[N_1/\delta\right]+1]$, 
we may assume that $\supp f_2$ is confined to $\mathbb{S}_\delta^{\xi^0(k)}$ with some 
fixed $\xi^0(k) \sim N_2$. Indeed, if $A$ is sufficiently large, 
from $L_{\textnormal{max}} \leq  A^{-1} N_1^2$, we get
\begin{align*}
& |\sigma_3 |\xi_1+\xi_2|^2 + \sigma_1 |\xi_1|^2 + \sigma_2 |\xi_2|^2| \leq 3A^{-1}N_1^2 \\
\Longrightarrow & -3A^{-1}N_1^2 \leq 
(\sigma_2 + \sigma_3) |\xi_2|^2 + 2 \sigma_3 |\xi_2| |\xi_1| \cos \theta_{12} + (\sigma_1 + \sigma_3) |\xi_1|^2 \leq 
3A^{-1}N_1^2 \\
\Longrightarrow &  
\begin{cases} &  \sigma_3 |\xi_1| \cos \theta_{12}
- |\xi_1| \sqrt{|\theta| - \sigma_3^2 \sin^2 \theta_{12} + 2^4 A^{-1}(\sigma_2 + \sigma_3)} \\
& \qquad \leq 
(\sigma_2+\sigma_3) |\xi_2| \leq 
 \sigma_3 |\xi_1| \cos \theta_{12} 
- |\xi_1| \sqrt{|\theta| - \sigma_3^2 \sin^2 \theta_{12} - 2^4 A^{-1}(\sigma_2 + \sigma_3)}, \\
&{\rm or}\\
& \sigma_3 |\xi_1| \cos \theta_{12}
+ |\xi_1| \sqrt{|\theta| - \sigma_3^2 \sin^2 \theta_{12} - 2^4 A^{-1}(\sigma_2 + \sigma_3)} \\
& \qquad \leq 
(\sigma_2+\sigma_3) |\xi_2| \leq 
 \sigma_3 |\xi_1| \cos \theta_{12} 
+ |\xi_1| \sqrt{|\theta| - \sigma_3^2 \sin^2 \theta_{12} + 2^4 A^{-1}(\sigma_2 + \sigma_3)},
\end{cases}
\end{align*}
where $\theta_{12} := \angle(\xi_1,\xi_2).$ Since $|\theta_{12}|$ is confined to a set of measure $\sim A^{-1}$, 
this suggests that 
if $(\ta_1, \xi_1) \in {\mathbf 1}_{\mathbb{S}_{\delta}^{N_1 + k \delta} } $ with fixed $k \in [0, \left[N_1/\delta\right]+1]$, $|\xi_2|$ is restricted to a set of measure $\sim \delta$. 
While, by symmetry, if $(\ta_2, \xi_2) \in {\mathbf 1}_{\mathbb{S}_{\delta}^{\xi^0 (k)} } $ with fixed $k \in [0, \left[N_1/\delta\right]+1]$, $|\xi_1|$ is confined to a set of measure $\sim \delta$. 
Thus we can assume that $f_1$ and $f_2$ in \eqref{0609-2} satisfy 
$\supp f_1 \subset \mathbb{S}_{\delta}^{N_1 + k \delta}$ 
and $\supp f_2 \subset \mathbb{S}_\delta^{\xi^0(k)}$ with fixed $k$. 
Furthermore, we apply a harmless decomposition to $f_1$, $f_2$, $f_3$ and assume that 
there exist $\xi^0_{f_1}$, $\xi^0_{f_2}$, $\xi^0_{f_3} \in \R^2$ such that 
$\supp f_1 \subset C_{A^{-1} N_1}(\xi_{f_1}^0)$, 
$\supp f_2 \subset C_{A^{-1} N_1}(\xi_{f_2}^0)$, 
$\supp f_3 \subset C_{A^{-1} N_1}(\xi_{f_3}^0)$ where 
\begin{equation*}
C_{\delta'}(\xi') := \{ (\ta, \xi) \in \R^3 \ | \ | \xi -\xi' | \leq \delta' \} \quad 
\textnormal{with some} \ \delta' >0  .
\end{equation*}
We apply the same strategy as 
that of the proof of 
Proposition 4.4 in \cite{BHHT09}. 
Applying the transformation $\ta_1 = - \sigma_1 |\xi_1|^2 + c_1$ and $\ta_2 = - \sigma_2 |\xi_2|^2 + c_2$ and Fubini's theorem, we find that it suffices to prove
\begin{equation}
\begin{split}
 \left| \int  f_1 (\phi_{c_1}^{\sigma_1} (\xi_1)) f_2 (\phi_{c_2}^{\sigma_2} (\xi_2)) 
f_3 (\phi_{c_1}^{\sigma_1} (\xi_1) + \phi_{c_2}^{\sigma_2} (\xi_2))  d\xi_1 d \xi_2 \right| & \\ \lesssim 
A^{\frac{1}{2}} N_1^{-1} 
\|f_1 \circ \phi_{c_1}^{\sigma_1} \|_{L_\xi^2} & \|f_2 \circ \phi_{c_2}^{\sigma_2} \|_{L_\xi^2}  \|f_3 \|_{L^2_{\tau \xi}}\label{2017-06-10a}
\end{split}
\end{equation}
where $f_3(\ta, \xi)$ is supported in $c_0 \leq \ta - \sigma_3 |\xi|^2 \leq c_0 +1$ and 
\begin{equation*}
\phi_{c_1}^{\sigma_1} (\xi) = (-\sigma_1 |\xi|^2 + c_1, \xi), 
\quad \phi_{c_2}^{\sigma_2} (\xi) = (-\sigma_2 |\xi|^2 + c_2, \xi).
\end{equation*}
We use the scaling $(\ta, \, \xi) \to (N_1^2 \ta , \, N_1 \xi)$ to define
\begin{equation*}
\tilde{f_1} (\ta_1, \xi_1) = f_1 (N_1^2 \ta_1, N_1 \xi_1), \quad \tilde{f_2} (\ta_2, \xi_2) = f_2 (N_1^2 \ta_2, N_1 \xi_2), 
\quad \tilde{f_3} (\ta, \xi) = f_3(N_1^2 \ta, N_1 \xi).
\end{equation*}
If we set ${\tilde{c_k}} = N_1^{-2} c_k$, inequality \eqref{2017-06-10a} reduces to
\begin{equation}
\begin{split}
 \left| \int  \tilde{f_1} (\phi_{\tilde{c_1}}^{\sigma_1} (\xi_1)) \tilde{f_2} (\phi_{\tilde{c_2}}^{\sigma_2} (\xi_2)) 
\tilde{f_3} (\phi_{\tilde{c_1}}^{\sigma_1} (\xi_1) + \phi_{\tilde{c_2}}^{\sigma_2} (\xi_2))  d\xi_1 d \xi_2 \right| & \\ \lesssim 
A^{\frac{1}{2}} N_1^{-1} 
\|\tilde{f_1} \circ \phi_{\tilde{c_1}}^{\sigma_1} \|_{L_\xi^2} & \|\tilde{f_2} \circ \phi_{\tilde{c_2}}^{\sigma_2} \|_{L_\xi^2}  \|\tilde{f_3} \|_{L^2_{\tau \xi}}\label{2017-06-10b}
\end{split}
\end{equation}
Note that $\tilde{f_3}$ is supported in $S_3(N_1^{-2})$ where 
\begin{equation*}
S_3 (N_1^{-2}) = \left\{ (\ta, \xi) \in C_{A^{-1}}(N_1^{-1} \xi_{f_3}^0) 
\ | \ \sigma_3 |\xi|^2 +\frac{c_0}{N_1^2} \leq \ta  \leq \sigma_3 |\xi|^2 +\frac{c_0+1}{N_1^2} \right\}.
\end{equation*}
By density and duality it suffices to show for continuous 
$\tilde{f_1}$ and $\tilde{f_2}$ that
\begin{equation}
\| \tilde{f_1} |_{S_1} * \tilde{f_2} |_{S_2} \|_{L^2(S_3 (N_1^{-2}))} 
\lesssim A^{\frac{1}{2}} N_1^{-1} 
\| \tilde{f_1} \|_{L^2(S_1)} \| \tilde{f_2} \|_{L^2(S_2)}\label{2017-06-10c}
\end{equation}
where $S_1$, $S_2$ denote the following surfaces 
\begin{align*}
S_1 = \{ \phi_{\tilde{c_1}}^{\sigma_1} (\xi_1)\in 
C_{ A^{-1}}(N_1^{-1} \xi_{f_1}^0) \}, \\
S_2 = \{ \phi_{\tilde{c_2}}^{\sigma_2} (\xi_2)\in 
C_{A^{-1}}(N_1^{-1} \xi_{f_2}^0) \}.
\end{align*}
\eqref{2017-06-10c} is immediately obtained by
\begin{equation}
\| \tilde{f_1} |_{S_1} * \tilde{f_2} |_{S_2} \|_{L^2(S_3)} 
\lesssim A^{\frac{1}{2}}  
\| \tilde{f_1} \|_{L^2(S_1)} \| \tilde{f_2} \|_{L^2(S_2)}\label{2017-06-10d}
\end{equation}
where 
\begin{equation*}
S_3 = \{ (\psi (\xi), \xi) \in C_{A^{-1}}(N_1^{-1} \xi_{f_3}^0) 
\ | \ \psi(\xi) =  \sigma_3 |\xi|^2 +\frac{c_0}{N_1^2} \}.
\end{equation*}
Since $|N_1^{-1} \xi_{f_1}^0| \sim |N_1^{-1} \xi_{f_2}^0| \sim |N_1^{-1} \xi_{f_3}^0| \sim 1$, after suitable harmless decomposition, we can assume 
\begin{equation}
\textnormal{diam}(S_k) \leq 2^{-10} |\theta| M  A^{-1},
\qquad (k=1,2,3).\label{2017-06-10e}
\end{equation}
Here we used the harmless constant
\begin{equation*}
M := \LR{\sigma_1}^{-2} \LR{\sigma_2}^{-2} \LR{\sigma_3}^{-2} 
\min (1, |\sigma_1 + \sigma_2|, \ |\sigma_2 + \sigma_3|, \ |\sigma_3 + \sigma_1|).
\end{equation*}
For any $\lambda_i \in S_i$, $i=1,2,3$, there exist $\xi_1$, $\xi_2$, $\xi$ such that
\begin{equation*}
\lambda_1=\phi_{{{\tilde{c_1}}}}^{\sigma_1} (\xi_1), \quad \lambda_2 =  \phi_{{{\tilde{c_2}}}}^{\sigma_2} (\xi_2), \quad \lambda_3 = (\psi (\xi), \xi),
\end{equation*}
and the unit normals ${\mathfrak{n}}_i$ on $\lambda_i$ are written as
\begin{align*}
& {\mathfrak{n}}_1(\lambda_1) = \frac{1}{\LR{2 {\sigma_1} |\xi_1|}} 
\left(1, \ 2 {\sigma_1} \xi_1^{(1)}, \ 2 {\sigma_1} \xi_1^{(2)} \right), \\
& {\mathfrak{n}}_2(\lambda_2) = \frac{1}{\LR{2 {\sigma_2} |\xi_2|}} 
\left(1, \ 2 {\sigma_2} \xi_2^{(1)}, \ 2 {\sigma_2} \xi_2^{(2)} \right), \\
& {\mathfrak{n}}_3(\lambda_3) = \frac{1}{\LR{2 {\sigma_3} |\xi|}} \ 
\left(-1, \ 2 {\sigma_3} \xi^{(1)}, \ 2 {\sigma_3} \xi^{(2)} \right),
\end{align*}
where $\xi^{(i)}$ $(i=1,2)$ denotes the $i$-th component of $\xi$. 
Clearly, the surfaces $S_1$, $S_2$, $S_3$ satisfy the following 
H\"{o}lder condition.
\begin{equation}
\sup_{\lambda_i, \lambda_i' \in S_i} \frac{|\mathfrak{n}_i(\lambda_i) - 
\mathfrak{n}_i(\lambda_i')|}{|\lambda_i - \lambda_i'|}
+ \frac{|\mathfrak{n}_i(\lambda_i) (\lambda_i - \lambda_i')|}{|\lambda_i - \lambda_i'|^2} \leq 2^3.\label{aiueo1}
\end{equation}
We may assume that there exist $\xi_1', \xi_2', \xi' \in \R^2$ such that
\begin{equation*}
\xi_1' + \xi_2' = \xi', \quad \phi_{{{\tilde{c_1}}}}^{\sigma_1} (\xi_1') \in S_1, \ \phi_{{{\tilde{c_2}}}}^{\sigma_2} (\xi_2') \in S_2, \ (\psi (\xi'), \xi') \in S_3,
\end{equation*}
otherwise the left-hand side of \eqref{2017-06-10d} vanishes. 
Let $\lambda_1' = \phi_{{{\tilde{c_1}}}}^{\sigma_1}(\xi_1')$, $\lambda_2' = \phi_{{{\tilde{c_2}}}}^{\sigma_2} (\xi_2')$, $\lambda_3' =  (\psi (\xi'), \xi') $. 
For any $\lambda_1 = \phi_{{{\tilde{c_1}}}}^{\sigma_1} (\xi_1) \in S_1$, 
we deduce from $\lambda_1$, $\lambda_1' \in S_1$ 
and \eqref{2017-06-10e} that
\begin{equation}
|{\mathfrak{n}}_1(\lambda_1) - {\mathfrak{n}}_1(\lambda_1')| \leq 2^{-7} |\theta| 
M A^{-1}.\label{2017-06-10f}
\end{equation}
Similarly, for any $\lambda_2 \in S_2$ and $\lambda_3 \in S_3$ we have
\begin{align}
& |{\mathfrak{n}}_2(\lambda_2) - {\mathfrak{n}}_2(\lambda_2')| \leq 2^{-7} |\theta| 
M  A^{-1}.\label{2017-06-10g}\\
& |{\mathfrak{n}}_3(\lambda_3) - {\mathfrak{n}}_3(\lambda_3')| \leq 2^{-7} |\theta| 
M A^{-1}.\label{2017-06-10h}
\end{align}
From \eqref{2017-06-10e}-\eqref{aiueo1}, 
once the following transversality condition \eqref{trans} is verified, we obtain the desired estimate \eqref{2017-06-10d}  by applying Proposition \ref{prop2.7} with $d = |\theta| 
M A^{-1}/2$.
\begin{equation}
\frac{|\theta|}{2} 
 M  A^{-1} \leq |\textnormal{det} N(\lambda_1, \lambda_2, \lambda_3)| \quad 
\textnormal{for any} \ \lambda_i \in S_i.\label{trans}
\end{equation}
From \eqref{2017-06-10f}-\eqref{2017-06-10h} it suffices to show
\begin{equation}
|\theta|  
M A^{-1} \leq |\textnormal{det} N(\lambda_1', \lambda_2', \lambda_3')| .
\end{equation}
Seeing that $\lambda_1' = \phi_{\tilde{c_1}}^{\sigma_1} (\xi_1')$, $\lambda_2' = \phi_{\tilde{c_2}}^{\sigma_2} (\xi_2')$, $\lambda_3' =  (\psi(\xi'), \xi') $ and 
$\xi_1'+\xi_2' = \xi'$, we get
\begin{align*}
|\textnormal{det} N(\lambda_1', \lambda_2', \lambda_3')| \geq & 
\frac{1}{\LR{2 {\sigma_3} |\xi'|}} 
\frac{1}{\LR{2 {\sigma_1} |\xi_1'|}} \frac{1}{\LR{2 {\sigma_2} |\xi_2'|}}  \left|\textnormal{det}
\begin{pmatrix}
1 & 1 & - 1 \\
 {\sigma_1} \xi_1'^{(1)}  &  {\sigma_2} \xi_2'^{(1)}  & {\sigma_3} \xi'^{(1)} \\
 {\sigma_1} \xi_1'^{(2)}   & {\sigma_2} \xi_2'^{(2)}  & {\sigma_3} \xi'^{(2)} 
\end{pmatrix} \right| \notag \\
\geq & \frac{1}{8} 
|\theta|  
M
 \left| \frac{\xi_1'^{(1)}  \xi_2'^{(2)}  - \xi_1'^{(2)} \xi_2'^{(1)}}{|\xi_1'||\xi_2'|} \right|  \notag\\
\geq & |\theta|  
M A^{-1}.
\end{align*}
\end{proof}
\subsection{The estimates for low modulation, $3$D}
Similarly to $2$D, we will utilize the operators with respect to angular variables. The following operators were introduced in \cite{BH11}. 
\begin{defn}[\cite{BH11}]
For each $A \in \N$, $\{\omega_A^j \}_{j \in \Omega_A} $ denotes a set of spherical caps of ${\BBB S}^2$ with the following properties:\\
(i) The angle $\angle{(x,y)}$ between 
any two vectors in $x$, $y \in \omega_A^j$ satisfies
\begin{equation*}
\left| \angle{(x,y)} \right| \leq A^{-1}.
\end{equation*}
(ii) Characteristic functions $\{ {\mathbf 1}_{\omega_A^j} \}$ satisfy
\begin{equation*}
1 \leq \sum_{j \in \Omega_A} {\mathbf 1}_{\omega_A^j}(x) \leq 3, \qquad \forall x \in {\BBB S}^2.
\end{equation*}

We define the function
\begin{equation*}
\alpha (j_1,j_2) = \inf \left\{ \left| \angle{( \pm x, y)} \right| : \ x \in \omega_A^{j_1}, \ y \in \omega_A^{j_2} \right\}
\end{equation*}
which measures the minimal angle between any two straight lines through the spherical caps $\omega_A^{j_1}$ and 
$\omega_A^{j_2}$, respectively. It is easily observed that for any fixed $j_1 \in \Omega_A$ there exist only a finite number of $j_2 \in \Omega_A$ which satisfies $\alpha (j_1,j_2) \sim A^{-1}$.

Based on the above construction, for each $j \in \Omega_A$ we define 
\begin{equation*}
{\mathfrak{D}}_{j}^A = \left\{ (\ta, \xi) \in \R \cross (\R^3 \setminus \{0\}) \, : \, 
 \frac{\xi}{|\xi|} \in \omega_A^j \right\}
\end{equation*}
and the corresponding localization operator
\begin{equation*}
\F (R_j^A u) (\tau, \xi) = \frac{\chi_{\omega_j^A} (\frac{\xi}{|\xi|})}{\chi  (\frac{\xi}{|\xi|})} \F u (\tau , \xi).
\end{equation*}
\end{defn}
\begin{thm}\label{thm-0.3-3d}
Let $\displaystyle L_{\max} \ll |\theta|  N_{\min}^2$, 
$A \geq 1$ and 
$\alpha (j_1,j_2) \lesssim A^{-1}$. 
Then the following estimates holds:
\begin{align}
\begin{split}
\|Q_{L_3}^{-\sigma_3} P_{N_3}(R_{j_1}^A Q_{L_1}^{\sigma_1}P_{N_1}u_{1}\cdot 
R_{j_2}^A Q_{L_2}^{\sigma_2}P_{N_2}u_{2})\|_{L^{2}_{tx}} & \\
\lesssim A^{-1} N_1^{\frac{1}{2}} L_1^{\frac{1}{2}}L_2^{\frac{1}{2}} 
\|R_{j_1}^A Q_{L_1}^{\sigma_1}P_{N_1}u_{1}\|_{L^2_{tx}} &  \|R_{j_2}^A Q_{L_2}^{\sigma_2}P_{N_2}u_{2}\|_{L^2_{tx}},
\end{split}\label{bilinear-12-3d}\\
\begin{split}
\|R_{j_1}^A Q_{L_1}^{-\sigma_1} P_{N_1}(R_{j_2}^A Q_{L_2}^{\sigma_2}P_{N_2}u_{2}\cdot 
 Q_{L_3}^{\sigma_3}P_{N_3}u_{3})\|_{L^{2}_{tx}} & \\
\lesssim A^{-1} N_1^{\frac{1}{2}} L_2^{\frac{1}{2}}L_3^{\frac{1}{2}} 
\|R_{j_2}^A Q_{L_2}^{\sigma_2}P_{N_2}u_{2}\|_{L^2_{tx}} &  \|Q_{L_3}^{\sigma_3}P_{N_3}u_{3}\|_{L^2_{tx}},
\end{split}\label{bilinear-23-3d}\\
\begin{split}
\|R_{j_2}^A Q_{L_2}^{-\sigma_2} P_{N_2}( Q_{L_3}^{\sigma_3}P_{N_3}u_{3}\cdot 
 R_{j_1}^A Q_{L_1}^{\sigma_1}P_{N_1}u_{1})\|_{L^{2}_{tx}} & \\
\lesssim A^{-1} N_1^{\frac{1}{2}} L_3^{\frac{1}{2}}L_1^{\frac{1}{2}} \|Q_{L_3}^{\sigma_3}P_{N_3}u_{3}\|_{L^2_{tx}} & 
\|R_{j_1}^A Q_{L_1}^{\sigma_1}P_{N_1}u_{1}\|_{L^2_{tx}}.
\end{split}\label{bilinear-31-3d}
\end{align}
\end{thm}
\begin{proof}
Similarly to the proof of Theorem \ref{thm-0.3}, we assume that $A$ is sufficiently large and show only 
\eqref{bilinear-12-3d}. By Plancherel's theorem, \eqref{bilinear-12-3d} can be written as
\begin{equation}
\begin{split}
 \left\| \psi_{L_3} (\tau - \sigma_3 |\xi|^2) \psi_{N_3}(\xi) \int 
f_1 (\tau_1,\xi_1) f_2 (\tau-\tau_1, \xi-\xi_1)d\tau_1d\xi_1 \right\|_{L_{\tau \xi}^2} & \\ 
\lesssim 
A^{-1} N_1^{\frac{1}{2}} L_1^{\frac{1}{2}}L_2^{\frac{1}{2}} 
\|f_1\|_{L^2_{\tau \xi}} & \|f_2 \|_{L^2_{\tau \xi}},
\end{split}\label{bilinear-15-3d}
\end{equation}
where $f_{i}=\F_{tx} [R_{j_i}^A Q_{L_i}^{\sigma_i}P_{N_i}u_{i}]$ $(i=1,2)$. 
Let $\xi = ( \xi^{(1)}, \xi^{(2)}, \xi^{(3)}) \in \R^3$ and $\tilde{\xi} = ( \xi^{(1)}, \xi^{(2)}) \in \R^2$. 
Since $\supp f_i \subset {\mathfrak{D}}_{j_i}^A$, after applying rotation in space and suitable decomposition, we may assume that the supports of $f_1$ and $f_2$ are both contained in the following slab
\begin{equation*}
\Sigma_3 (N_1 A^{-1}) := \{ (\ta, \xi) \in \R \cross \R^3 \ | \ |\xi^{(3)}| \leq  N_1 A^{-1} \}.
\end{equation*}
Let $\psi_{N_3, L_3}^{\sigma_3}(\ta,\xi) := \psi_{L_3} (\tau - \sigma_3 |\xi|^2) \psi_{N_3}(\xi)$. 
It suffices to show
\begin{equation}
\begin{split}
& \left\| \int_{\R^3} \psi_{N_3, L_3}^{\sigma_3}(\ta,\xi)  
f_1 (\ta_1, \tilde{\xi_1}, \xi_1^{(3)}) \, 
f_2 (\ta - \ta_1, \tilde{\xi} - \tilde{\xi_1}, \xi^{(3)} - \xi_1^{(3)}) d\ta_1 d \tilde{\xi_1} \right\|_{L_{\tau \tilde{\xi}}^2}\\
& \qquad \qquad \qquad \qquad \qquad  \lesssim A^{-\frac{1}{2}}
(L_1 L_2)^{\frac{1}{2}}\|f_1 ( \xi_1^{(3)}) \|_{L_{\tau \tilde{\xi}}^2} 
\|f_2 (\xi^{(3)} - \xi_1^{(3)}) \|_{L_{\tau \tilde{\xi}}^2} 
\end{split}\label{2017-12-09}
\end{equation}
for any $|\xi^{(3)} - \xi_1^{(3)}| \leq  N_1 A^{-1}$ and $|\xi_1^{(3)}| \leq  N_1 A^{-1}$. Indeed, once  \eqref{2017-12-09} holds, 
from Minkowski's inequality and Young's inequality, we have
\begin{align*}
&  \left\|  \psi_{N_3, L_3}^{\sigma_3}(\ta,\xi) \int_{\R^4}  
f_1 (\tau_1,\xi_1) f_2 (\tau-\tau_1, \xi-\xi_1)d\tau_1d\xi_1 \right\|_{L_{\tau \xi}^2} \\
= \ & \left\|  \psi_{N_3, L_3}^{\sigma_3}(\ta,\xi) 
\int_{\R^4}  f_1 (\ta_1, \tilde{\xi_1}, \xi_1^{(3)})  \, 
f_2 (\ta - \ta_1, \tilde{\xi} - \tilde{\xi_1}, \xi^{(3)} - \xi_1^{(3)}) d\ta_1 d \xi_1 \right\|_{L_{\tau \xi}^2}\\
\lesssim \ & \left\| \int_\R \left\|  \psi_{N_3, L_3}^{\sigma_3}  \int_{\R^3} 
f_1 (\ta_1, \tilde{\xi_1}, \xi_1^{(3)}) \, 
f_2 (\ta - \ta_1, \tilde{\xi} - \tilde{\xi_1}, \xi^{(3)} - \xi_1^{(3)}) d\ta_1 d \tilde{\xi_1} \right\|_{L_{\tau \tilde{\xi}}^2}  
d\xi_1^{(3)} 
\right\|_{L_{\xi^{(3)} }^2}\\
 \underset{\eqref{2017-12-09}}{\lesssim} & A^{-\frac{1}{2}}
(L_1 L_2)^{\frac{1}{2}} \left\| \int_\R \|f_1  ( \xi_1^{(3)}) \|_{L_{\tau \tilde{\xi}}^2} 
\|f_2 (\xi^{(3)} - \xi_1^{(3)}) \|_{L_{\tau \tilde{\xi}}^2} d\xi_1^{(3)} 
\right\|_{L_{\xi^{(3)} }^2}\\
\lesssim \ &  A^{-1} N_1^{\frac{1}{2}} 
(L_1 L_2)^{\frac{1}{2}}\sup_{\xi^{(3)}} \int_\R \|f_1  ( \xi_1^{(3)} ) \|_{L_{\tau \tilde{\xi}}^2} 
\|f_2 (\xi^{(3)} - \xi_1^{(3)}) \|_{L_{\tau \tilde{\xi}}^2} d\xi_1^{(3)} \\
\lesssim \ &  A^{- 1} N_1^{\frac{1}{2}} ( L_1 L_2 )^{\frac{1}{2}} \|f_1 \|_{ L_{\tau \xi}^2} 
\|f_2 \|_{ L_{\ta \xi}^2}.
\end{align*}
Since the estimate \eqref{2017-12-09} can be verified by the same proof as that of Theorem \ref{thm-0.3}, we omit the details.
\end{proof}
\begin{prop}\label{thm2.6-3d}
Let $L_{\textnormal{max}} \ll |\theta| N_{\min}^2$
, $ 64 \leq A \leq N_{\textnormal{max}}$ and 
$\alpha (j_1,j_2) \lesssim A^{-1}$. 
Then the following estimate holds:
\begin{equation}
\begin{split}
\|Q_{L_3}^{-\sigma_3} P_{N_3}(R_{j_1}^A Q_{L_1}^{\sigma_1}P_{N_1}u_{1}\cdot 
R_{j_2}^A Q_{L_2}^{\sigma_2}P_{N_2}u_{2})\|_{L^{2}_{tx}} & \\
\lesssim  N_1^{-\frac{1}{2}} L_1^{\frac{1}{2}}L_2^{\frac{1}{2}} L_3^{\frac{1}{2}} 
\|R_{j_1}^A Q_{L_1}^{\sigma_1}P_{N_1}u_{1}\|_{L^2_{tx}} &  \|R_{j_2}^A Q_{L_2}^{\sigma_2}P_{N_2}u_{2}\|_{L^2_{tx}}.
\end{split}\label{0609-3d}
\end{equation}
\end{prop}
\begin{proof}
We use the same notations as in the proof of Theorem \ref{thm-0.3-3d}. Applying Fourier transform, we rewrite \eqref{0609-3d} as
\begin{equation}
\begin{split}
 \left\| \psi_{N_3, L_3}^{\sigma_3}(\ta,\xi) \int 
f_1 (\tau_1,\xi_1) f_2 (\tau-\tau_1, \xi-\xi_1)d\tau_1d\xi_1 \right\|_{L_{\tau \xi}^2} & \\ 
\lesssim 
A^{-1} N_1^{\frac{1}{2}} L_1^{\frac{1}{2}}L_2^{\frac{1}{2}} 
\|f_1\|_{L^2_{\tau \xi}} & \|f_2 \|_{L^2_{\tau \xi}}.
\end{split}\label{bilinear-25-3d}
\end{equation}
Similarly to the proof of Theorem \ref{thm-0.3-3d}, we assume that the supports of $f$ and $g$ are both contained in the slab $\Sigma_3  (N_1 A^{-1})$. 
Thanks to $|\xi|^2 = |\tilde{\xi}|^2 + (\xi^{(3)})^2$, applying the same argument as in the proof of Theorem \ref{thm2.6}, we may obtain
\begin{equation}
\begin{split}
& \left\| \int \psi_{N_3, L_3}^{\sigma_3}(\ta,\xi) f_1(\ta_1, \tilde{\xi_1}, (\xi_1)_3) 
f_2 (\tau - \tau_1, \tilde{\xi} - \tilde{\xi_1}, \xi^{(3)} - \xi_1^{(3)} ) d\ta_1  
d\tilde{\xi_1}  \right\|_{L_{\tau \tilde{\xi}}^2}  \\
& \qquad \qquad \qquad \qquad \qquad  \lesssim 
A^{\frac{1}{2}} N_1^{-1} (L_0 L_1 L_2)^\frac{1}{2}
  \|f_1 (\xi_1^{(3)}) \|_{L_{\tau \tilde{\xi}}^2}  
\|f_2 (\xi^{(3)}- \xi_1^{(3)}) \|_{L_{\tau \tilde{\xi}}^2}
\end{split}\label{2017-12-12}
\end{equation}
for any $|\xi^{(3)} - \xi_1^{(3)}| \leq  N_1 A^{-1}$ and $|\xi_1^{(3)}| \leq  N_1 A^{-1}$. 
To avoid redundancy, we omit the proof of \eqref{2017-12-12}. 
From \eqref{2017-12-12} and Minkowski's inequality, we get
\begin{align*}
&  \left\| \psi_{N_3, L_3}^{\sigma_3}(\ta,\xi) \int 
f_1 (\tau_1,\xi_1) f_2 (\tau-\tau_1, \xi-\xi_1)d\tau_1d\xi_1 \right\|_{L_{\tau \xi}^2} \\
\lesssim \ & \left\| \int_\R \left\|  \psi_{N_3, L_3}^{\sigma_3}  \int_{\R^3} 
f_1 (\ta_1, \tilde{\xi_1}, \xi_1^{(3)}) \, 
f_2 (\ta - \ta_1, \tilde{\xi} - \tilde{\xi_1}, \xi^{(3)} - \xi_1^{(3)}) d\ta_1 d \tilde{\xi_1} \right\|_{L_{\tau \tilde{\xi}}^2}  
d\xi_1^{(3)} 
\right\|_{L_{\xi^{(3)} }^2}\\
 \underset{\eqref{2017-12-12}}{\lesssim} & A^{\frac{1}{2}} N_1^{-1} 
(L_1 L_2 L_3)^{\frac{1}{2}} \left\| \int_\R \|f_1  ( \xi_1^{(3)}) \|_{L_{\tau \tilde{\xi}}^2} 
\|f_2 (\xi^{(3)} - \xi_1^{(3)}) \|_{L_{\tau \tilde{\xi}}^2} d\xi_1^{(3)} 
\right\|_{L_{\xi^{(3)} }^2}\\
\lesssim \ & N_1^{-\frac{1}{2}}   
(L_1 L_2 L_3)^{\frac{1}{2}} 
 \sup_{\xi^{(3)}} \int_\R \|f_1 ( \xi_1^{(3)} ) \|_{L_{\tau \tilde{\xi}}^2} 
\|f_2 (\xi^{(3)} - \xi_1^{(3)}) \|_{L_{\tau \tilde{\xi}}^2} d\xi_1^{(3)} \\
\lesssim \ &  N_1^{-\frac{1}{2}} 
(L_1 L_2 L_3)^{\frac{1}{2}} \|f_1 \|_{ L_{\tau \xi}^2} 
\|f_2 \|_{ L_{\ta \xi}^2}
\end{align*}
which completes the proof of \eqref{bilinear-25-3d}.
\end{proof}
\subsection{Proof of Proposition~\ref{key_be}}
We now prove the key estimate Proposition \ref{key_be}. 
\begin{proof}[Proof of Proposition~\ref{key_be}]
By the duality argument, it suffice to show that
\[
\begin{split}
&\sum_{N_1,N_2,N_3\ge 1}\sum_{L_1,L_2, L_3 \ge 1} N_{\max}
\left|\int u_{N_1,L_1}v_{N_2,L_2}w_{N_3,L_3}dxdt\right|\\
&\lesssim 
\|u\|_{X^{s,b',1}_{\sigma_1}} \|v\|_{X^{s,b',1}_{\sigma_2}} \|w\|_{X^{-s, b', 1}_{\sigma_3}}
\end{split}
\]
for the scalar functions $u$, $v$, and $w$, where we put
\[
u_{N_1,L_1}:=Q_{L_1}^{\sigma_1}P_{N_1}u,\ 
v_{N_2,L_2}:=Q_{L_2}^{\sigma_2}P_{N_2}v,\ 
w_{N_3,L_3}:=Q_{L_3}^{\sigma_3}P_{N_3}w
\]
and used $(Q_{L_3}^{-\sigma_3}f,\overline{g})_{L^2_{tx}}=(f,\overline{Q_{L_3}^{\sigma_3}g})_{L^2_{tx}}$. 
By Plancherel's theorem, we have
\[
\begin{split}
&\left|\int u_{N_1,L_1}v_{N_2,L_2}w_{N_3,L_3}dxdt\right|\\
&\sim 
\left|\int_{\substack{\xi_1+\xi_2+\xi_3=0\\ \tau_1+\tau_2+\tau_3=0}}
\F_{tx}[u_{N_1,L_1}](\tau_1,\xi_1)\F_{tx}[v_{N_2,L_2}](\tau_2,\xi_2)\F_{tx}[w_{N_3,L_3}](\tau_3,\xi_3)\right|.
\end{split}
\]
Thus, it is clear that we only need to consider the following three cases:
\begin{equation*}
\textnormal{(I)} \ N_1 \lesssim N_2 \sim N_3, \quad 
\textnormal{(I\hspace{-0.2mm}I)} \ N_2 \lesssim N_3 \sim N_1, \quad 
\textnormal{(I\hspace{-0.2mm}I\hspace{-0.2mm}I)} \ N_3 \lesssim N_1 \sim N_2.
\end{equation*}
To avoid redundancy, we only consider the first case. The other two can be shown similarly. 
It suffices to show that
\begin{equation}\label{desired_est}
\begin{split}
&N_{2}
\left|\int u_{N_1,L_1}v_{N_2,L_2}w_{N_3,L_3}dxdt\right|\\
&\lesssim
\left( \frac{N_{1}}{N_{2}} \right)^{\epsilon}N_{1}^s(L_1L_2L_3)^{b'}
\|u_{N_1,L_1}\|_{L^2_{tx}}\|v_{N_2,L_2}\|_{L^2_{tx}}\|w_{N_3,L_3}\|_{L^2_{tx}}
\end{split}
\end{equation}
for some $b'\in (0,\frac{1}{2})$ and $\epsilon >0$. 
Indeed, from (\ref{desired_est}) and the Cauchy-Schwartz inequality, 
we obtain 
\[
\begin{split}
&
\sum_{N_1 \lesssim N_2 \sim N_3}\sum_{L_1,L_2, L_3 \ge 1} N_{2}
\left|\int u_{N_1,L_1}v_{N_2,L_2}w_{N_3,L_3}dxdt\right|\\
&\lesssim 
 \sum_{N_1 \lesssim N_2 \sim N_3}\sum_{L_1,L_2, L_3\ge 1}
\left( \frac{N_{1}}{N_{2}} \right)^{\epsilon}N_{1}^s(L_1L_2L_3)^{b'}
\|u_{N_1,L_1}\|_{L^2_{tx}}\|v_{N_2,L_2}\|_{L^2_{tx}}\|w_{N_3,L_3}\|_{L^2_{tx}}\\
&\lesssim  \sum_{N_3}\sum_{ N_2 \sim N_3}
\left( \sum_{N_1\lesssim N_2}N_1^{s+\e} N_2^{-\e} \sum_{L_1 \ge 1}L_1^{b'}\|u_{N_1,L_1}\|_{L^2_{tx}}\right) 
\sum_{L_2\ge 1}L_2^{b'}\|v_{N_2,L_2}\|_{L^2_{tx}}\sum_{L_3\ge 1}L_3^{b'}\|w_{N_3,L_3}\|_{L^2_{tx}} \\
&\lesssim \|u\|_{X^{s,b',1}_{\sigma_1}} 
\sum_{N_3}\sum_{ N_2 \sim N_3} \left( N_2^{s} \sum_{L_2\ge 1}L_2^{b'}\|v_{N_2,L_2}\|_{L^2_{tx}} \right) 
\left( N_3^{-s}  \sum_{L_3\ge 1}L_3^{b'}\|w_{N_3,L_3}\|_{L^2_{tx}} \right)\\
& \lesssim \|u\|_{X^{s,b',1}_{\sigma_1}} \|v\|_{X^{s,b',1}_{\sigma_2}} \|w\|_{X^{-s, b', 1}_{\sigma_3}}
\end{split}
\]
Hence, we focus on (\ref{desired_est}) for $N_1\lesssim N_2\sim N_3$. 
We put 
$\displaystyle L_{\max} := \max_{1\leq j\leq 3} (L_1, L_2, L_3)$.\\
\kuuhaku \\
\underline{Case\ 1:\ High modulation, $\displaystyle L_{\max}\gtrsim N_{\max}^2$}

We assume $L_1\gtrsim N_{\max}^2\sim N_2^2$. 
By the Cauchy-Schwartz inequality and (\ref{L2be_est_2}), 
we have
\[
\begin{split}
&\left|\int u_{N_1,L_1}v_{N_2,L_2}w_{N_3,L_3}dxdt\right|\\
&\lesssim \|u_{N_1,L_1}\|_{L^2_{tx}}\|P_{N_1}(v_{N_2,L_2}w_{N_3,L_3})\|_{L^2_{tx}}\\
&\lesssim N_1^{\frac{d}{2}-1+4\delta}
\left(\frac{N_1}{N_2}\right)^{\frac{1}{2}- 2\delta}L_2^{b'}L_3^{b'}
\|u_{N_1,L_1}\|_{L^2_{tx}}\|v_{N_2,L_2}\|_{L^2_{tx}}\|w_{N_3,L_3}\|_{L^2_{tx}}.
\end{split}
\]
Therefore, we obtain
\[
\begin{split}
&N_{2}
\left|\int u_{N_1,L_1}v_{N_2,L_2}w_{N_3,L_3}dxdt\right|\\
&\lesssim N_1^{\frac{d-1}{2}+2\delta} N_2^{\frac{1}{2}-2b'+2\delta}
(L_1L_2L_3)^{b'}
\|u_{N_1,L_1}\|_{L^2_{tx}}\|v_{N_2,L_2}\|_{L^2_{tx}}\|w_{N_3,L_3}\|_{L^2_{tx}}.
\end{split}
\]
Thus, it suffices to show that
\begin{equation}\label{wait_esti}
N_1^{\frac{d-1}{2}+2\delta} N_2^{\frac{1}{2}-2b'+2\delta}
\lesssim \left( \frac{N_{1}}{N_{2}} \right)^{\epsilon}N_{1}^s
\end{equation}
for some $\epsilon >0$. Since $\delta =\frac{1}{2}-b'$, we have
\[
\begin{split}
N_1^{\frac{d-1}{2}+2\delta} N_2^{\frac{1}{2}-2b'+2\delta}
&=N_1^{\frac{d+1}{2}- 2 b'} N_2^{\frac{3}{2}- 4b'} \\
&\sim N_1^{\frac{d+4}{2}- 6 b' -s}
\left( \frac{N_{1}}{N_{2}} \right)^{4b' - \frac{3}{2} }N_{1}^s
\end{split}
\]
If $d=2$, then we obtain
\[
\begin{split}
 N_1^{\frac{d+4}{2}- 6 b' -s}
\le N_1^{-6(b'-\frac{5}{12})}
\end{split}
\]
for $s\ge \frac{1}{2}$. 
Therefore, by choosing $b'\in [\frac{5}{12}, \frac{1}{2})$, 
we get (\ref{wait_esti}). 
While if $d=3$, then we obtain 
\[
\begin{split}
 N_1^{\frac{d+4}{2}- 6 b' -s}
= N_1^{-\left( \left( s-\frac{1}{2} \right) -6 \left( \frac{1}{2} - b' \right) \right) }.
\end{split}
\]
Therefore, by choosing $b'\in (\frac{1}{2}-\frac{1}{6}(s-\frac{1}{2}),\frac{1}{2})$ 
for $s>\frac{1}{2}$, we get (\ref{wait_esti}). 

The proofs for the cases $L_2\gtrsim N_{\max}^2$ and $L_3\gtrsim N_{\max}^2$ are quite same. We omit them.\\
\kuuhaku \\
\underline{Case\ 2:\ Low modulation, $\displaystyle L_{\max}\ll N_{\max}^2$}

By Lemma \ref{modul_est_2}, we can assume $N_1 \sim N_2 \sim N_3$ thanks to $\displaystyle L_{\max}\ll N_{\max}^2$.\\
\noindent\textbf{$\circ 2$D}\\
We first consider $2$D case. Let $M_0 := (L_{\max})^{-\frac{1}{2}} N_{\max}\sim (L_{\max})^{-\frac{1}{2}} N_{1}$. We decompose $\R^3 \cross \R^3$ as follows:
\begin{equation*}
\R^3 \cross \R^3 =  \bigcup_{\tiny{\substack{0 \leq j_1,j_2 \leq M_0 -1\\|j_1 - j_2|\leq 16}}} 
{\mathfrak{D}}_{j_1}^{M_0} \cross {\mathfrak{D}}_{j_2}^{M_0} \cup \bigcup_{64 \leq A \leq M_0} \ \bigcup_{\tiny{\substack{0 \leq j_1,j_2 \leq A -1\\ 16 \leq |j_1 - j_2|\leq 32}}} 
{\mathfrak{D}}_{j_1}^A \cross {\mathfrak{D}}_{j_2}^A.
\end{equation*}
We can write
\begin{align*}
 \left|\int u_{N_1,L_1}v_{N_2,L_2}w_{N_3,L_3}dxdt\right| &
 \leq \sum_{{\tiny{\substack{A=M_0\\-M_0 \leq j_1,j_2 \leq M_0 -1\\|j_1 - j_2|\leq 16}}}} 
\left|\int u_{N_1,L_1, j_1}v_{N_2,L_2, j_2}w_{N_3,L_3}dxdt\right| \\ & + 
\sum_{64 \leq A \leq M_0} \sum_{{\tiny{\substack{-A \leq j_1,j_2 \leq A-1\\|j_1 - j_2|\leq 16}}}} 
\left|\int u_{N_1,L_1, j_1}v_{N_2,L_2, j_2}w_{N_3,L_3}dxdt\right| 
\end{align*}
with $u_{N_1,L_1, j_1} := R_{j_1}^A u_{N_1, L_1}$ and 
$v_{N_2,L_2, j_2} := R_{j_2}^A v_{N_2, L_2}$. 
We assume $L_{\textnormal{max}} = L_3$ for 
simplicity. The other cases can be treated similarly. For the former term, by using the H\"older inequality and Theorem \ref{thm-0.3}, we get
\begin{align*}
& \sum_{{\tiny{\substack{-M_0 \leq j_1,j_2 \leq M_0 -1\\|j_1 - j_2|\leq 16}}}} 
\left|\int u_{N_1,L_1, j_1}v_{N_2,L_2, j_2}w_{N_3,L_3}dxdt\right| \\
& \lesssim  \|w_{N_3,L_3} \|_{{L^2_{t x}}}  
\sum_{{\tiny{\substack{-M_0 \leq j_1,j_2 \leq M_0 -1\\|j_1 - j_2|\leq 16}}}}  
\|Q_{L_3}^{-\sigma_3} P_{N_3} ( u_{N_1,L_1, j_1} v_{N_2,L_2, j_2})\|_{L^{2}_{tx}}\\
& \lesssim 
\left(L_3^{-\frac{1}{2}} N_1 \right)^{-\frac{1}{2}} L_1^{\frac{1}{2}}L_2^{\frac{1}{2}} \|w_{N_3,L_3} \|_{{L^2_{t x}}}
\sum_{{\tiny{\substack{-M_0 \leq j_1,j_2 \leq M_0 -1\\|j_1 - j_2|\leq 16}}}} 
\|u_{N_1,L_1, j_1}\|_{L^2_{tx}}   \|v_{N_2,L_2, j_2}\|_{L^2_{tx}}\\
& \lesssim  N_1^{-\frac{1}{2}} ( L_1 L_2 L_3)^{\frac{5}{12}} \|u_{N_1,L_1}\|_{L^2_{tx}}\|v_{N_2,L_2}\|_{L^2_{tx}}\|w_{N_3,L_3}\|_{L^2_{tx}}.
\end{align*}
For the latter term, it follows from Proposition \ref{thm2.6} that we get
\begin{align*}
& \sum_{64 \leq A \leq M_0} \sum_{{\tiny{\substack{-A \leq j_1,j_2 \leq A-1\\|j_1 - j_2|\leq 16}}}} 
\left|\int u_{N_1,L_1, j_1}v_{N_2,L_2, j_2}w_{N_3,L_3}dxdt\right| \\
& \lesssim   \|w_{N_3,L_3} \|_{{L^2_{t x}}} 
\sum_{64 \leq A \leq M_0} \sum_{{\tiny{\substack{-A \leq j_1,j_2 \leq A-1\\|j_1 - j_2|\leq 16}}}} 
\|Q_{L_3}^{-\sigma_3} P_{N_3} ( u_{N_1,L_1, j_1} v_{N_2,L_2, j_2})\|_{L^{2}_{tx}}\\
& \lesssim \|w_{N_3,L_3} \|_{{L^2_{t x}}}
\sum_{64 \leq A \leq M_0}  
A^{\frac{1}{2}} N_1^{-1} ( L_1 L_2 L_3)^\frac{1}{2}
 \sum_{{\tiny{\substack{-A \leq j_1,j_2 \leq A-1\\|j_1 - j_2|\leq 16}}}}  
\|u_{N_1,L_1, j_1}\|_{L^2_{tx}}   \|v_{N_2,L_2, j_2}\|_{L^2_{tx}}\\
&  \lesssim  N_1^{-\frac{1}{2}} ( L_1 L_2 L_3)^{\frac{5}{12}} \|u_{N_1,L_1}\|_{L^2_{tx}}\|v_{N_2,L_2}\|_{L^2_{tx}}\|w_{N_3,L_3}\|_{L^2_{tx}}.
\end{align*}
The above two estimates give the desired estimate \eqref{desired_est}.\\
\noindent\textbf{$\circ 3$D}\\
Next, we consider $3$D case. The proof is almost the same as that for $2$D. We use Theorem  \ref{thm-0.3-3d} and Proposition \ref{thm2.6-3d} instead of 
Theorem  \ref{thm-0.3} and Proposition \ref{thm2.6}, respectively. 
Similarly to $2$D, we decompose $\R^4 \cross \R^4$ as follows:
\begin{equation*}
\R^4 \cross \R^4 = \bigcup_{\tiny{\substack{j_1,j_2 \in \Omega_{N_1} \\ \alpha (j_1,j_2)\sim N_1^{-1}}}} 
{{\mathfrak{D}}_{j_1}}^{N_1} \cross {{\mathfrak{D}}_{j_2}}^{N_1} \cup \bigcup_{64 \leq A \leq N_1} \ 
\bigcup_{\tiny{\substack{ j_1,j_2 \in \Omega_A \\ \alpha (j_1,j_2) \sim A^{-1}}}} 
{{\mathfrak{D}}_{j_1}}^A \cross {{\mathfrak{D}}_{j_2}}^A.
\end{equation*}
We can write
\begin{align*}
 \left|\int u_{N_1,L_1}v_{N_2,L_2}w_{N_3,L_3}dxdt\right| &
 \leq \sum_{{\tiny{\substack{A=N_1\\j_1,j_2 \in \Omega_{N_1} \\ \alpha (j_1,j_2)\sim N_1^{-1}}}}}  
\left|\int u_{N_1,L_1, j_1}v_{N_2,L_2, j_2}w_{N_3,L_3}dxdt\right| \\ & + 
\sum_{64 \leq A \leq N_1} \sum_{{\tiny{\substack{ j_1,j_2 \in \Omega_A \\ \alpha (j_1,j_2) \sim A^{-1}}}}} 
\left|\int u_{N_1,L_1, j_1}v_{N_2,L_2, j_2}w_{N_3,L_3}dxdt\right| 
\end{align*}
with $u_{N_1,L_1, j_1} := R_{j_1}^A u_{N_1, L_1}$ and 
$v_{N_2,L_2, j_2} := R_{j_2}^A v_{N_2, L_2}$. 
We assume $L_{\textnormal{max}} = L_3$ for 
simplicity. 
For the former term, by using the H\"older inequality and Theorem \ref{thm-0.3}, we get
\begin{align*}
& \sum_{{\tiny{\substack{j_1,j_2 \in \Omega_{N_1} \\ \alpha (j_1,j_2)\sim N_1^{-1}}}}} 
\left|\int u_{N_1,L_1, j_1}v_{N_2,L_2, j_2}w_{N_3,L_3}dxdt\right| \\
& \lesssim  \|w_{N_3,L_3} \|_{{L^2_{t x}}}  
\sum_{{\tiny{\substack{j_1,j_2 \in \Omega_{N_1}\\ \alpha (j_1,j_2)\sim N_1^{-1}}}}}  
\|Q_{L_3}^{-\sigma_3} P_{N_3} ( u_{N_1,L_1, j_1} v_{N_2,L_2, j_2})\|_{L^{2}_{tx}}\\
& \lesssim 
 N_1^{-\frac{1}{2}} L_1^{\frac{1}{2}}L_2^{\frac{1}{2}} \|w_{N_3,L_3} \|_{{L^2_{t x}}}
\sum_{{\tiny{\substack{j_1,j_2 \in \Omega_{N_1}\\ \alpha (j_1,j_2)\sim N_1^{-1}}}}} 
\|u_{N_1,L_1, j_1}\|_{L^2_{tx}}   \|v_{N_2,L_2, j_2}\|_{L^2_{tx}}\\
& \lesssim  N_1^{-\frac{1}{2}} ( L_1 L_2 L_3)^{\frac{1}{3}} \|u_{N_1,L_1}\|_{L^2_{tx}}\|v_{N_2,L_2}\|_{L^2_{tx}}\|w_{N_3,L_3}\|_{L^2_{tx}}.
\end{align*}
For the latter term, it follows from Proposition \ref{thm2.6} and $\displaystyle L_{\max}\ll N_{1}^2$ that we get
\begin{align*}
& \sum_{64 \leq A \leq N_1} \sum_{{\tiny{\substack{ j_1,j_2 \in \Omega_A \\ \alpha (j_1,j_2) \sim A^{-1}}}}} 
\left|\int u_{N_1,L_1, j_1}v_{N_2,L_2, j_2}w_{N_3,L_3}dxdt\right| \\
& \lesssim   \|w_{N_3,L_3} \|_{{L^2_{t x}}} 
\sum_{64 \leq A \leq N_1} \sum_{{\tiny{\substack{ j_1,j_2 \in \Omega_A \\ \alpha (j_1,j_2) \sim A^{-1}}}}} 
\|Q_{L_3}^{-\sigma_3} P_{N_3} ( u_{N_1,L_1, j_1} v_{N_2,L_2, j_2})\|_{L^{2}_{tx}}\\
& \lesssim \|w_{N_3,L_3} \|_{{L^2_{t x}}}
\sum_{64 \leq A \leq N_1}  
N_1^{-\frac{1}{2}} ( L_1 L_2 L_3)^\frac{1}{2}
\sum_{{\tiny{\substack{ j_1,j_2 \in \Omega_A \\ \alpha (j_1,j_2) \sim A^{-1}}}}}   
\|u_{N_1,L_1, j_1}\|_{L^2_{tx}}   \|v_{N_2,L_2, j_2}\|_{L^2_{tx}}\\
&  \lesssim  (\log{N_1}) N_1^{-\frac{1}{2}} ( L_1 L_2 L_3)^{\frac{1}{2}} \|u_{N_1,L_1}\|_{L^2_{tx}}\|v_{N_2,L_2}\|_{L^2_{tx}}\|w_{N_3,L_3}\|_{L^2_{tx}}\\
&  \lesssim   N_1^{s-1} (L_1 L_2 L_3)^{\frac{1}{2}-\frac{1}{10} \left( s-\frac{1}{2} \right)} \|u_{N_1,L_1}\|_{L^2_{tx}}\|v_{N_2,L_2}\|_{L^2_{tx}}\|w_{N_3,L_3}\|_{L^2_{tx}}.
\end{align*}
This completes the proof of \eqref{desired_est}.
\end{proof}
%
%
%
%
%
\section{Proof of the well-posedness}
In this section, we prove Theorems~\ref{wellposed_1}. 
First, we give the linear estimate. 
\begin{prop}\label{linear_est}
Let $s\in \R$, $\sigma\in \R\backslash \{0\}$, $b'\in (0,\frac{1}{2})$ and 
$0<T\le 1$. 
\begin{enumerate}
\item[(1)] There exists $C_1>0$ such that for any $\varphi \in H^s$, we have
\[
\|e^{it\sigma \Delta}\varphi\|_{X^{s,\frac{1}{2},1}_{\sigma,T}}
\le C_1\|\varphi\|_{H^s}.
\]
\item[(2)] There exists $C_2>0$ such that for any $F \in X^{s,-b',\infty}_{\sigma,T}$, we have
\[
\left\|\int_{0}^{t}e^{i(t-t')\sigma \Delta}F(t')dt'\right\|_{X^{s,\frac{1}{2},1}_{\sigma,T}}
\le C_2T^{\frac{1}{2}-b'}\|F\|_{X^{s,-b',\infty}_{\sigma,T}}.
\]
\item[(3)] There exists $C_3>0$ such that for any $u \in X^{s,\frac{1}{2},\infty}_{\sigma,T}$, we have
\[
\|u\|_{X^{s,b',1}_{\sigma, T}}\le C_3T^{\frac{1}{2}-b'}\|u\|_{X^{s,\frac{1}{2},1}_{\sigma, T}}. 
\]
\end{enumerate}
\end{prop}
For the proof of (1) and (2), see Proposition\ 5.3 in \cite{BHHT09}
 (and also Lemma\ 2.1 in \cite{GTV97}). 
For the proof of (3), see Proposition\ 5.3 in \cite{BHHT09}. 

We define the map $\Phi(u,v,w)=(\Phi_{\alpha, u_{0}}^{(1)}(w, v), \Phi_{\beta, v_{0}}^{(1)}(\overline{w}, v), \Phi_{\gamma, w_{0}}^{(2)}(u, \overline{v}))$ as
\[
\begin{split}
\Phi_{\sigma, \varphi}^{(1)}(f,g)(t)&:=e^{it\sigma \Delta}\varphi 
-\int_{0}^{t}e^{i(t-t')\sigma \Delta}(\nabla \cdot f(t'))g(t')dt',\\
\Phi_{\sigma, \varphi}^{(2)}(f,g)(t)&:=e^{it\sigma \Delta}\varphi 
+\int_{0}^{t}e^{i(t-t')\sigma \Delta}\nabla (f(t')\cdot g(t'))dt'.
\end{split}
\] 
To prove the existence of the solution of (\ref{NLS_sys}), we prove that $\Phi$ is a contraction map 
on $B_R(X^{s,\frac{1}{2},1}_{\alpha,T}\times X^{s,\frac{1}{2},1}_{\beta,T}\times X^{s,\frac{1}{2},1}_{\gamma,T})$ for some $R>0$ and $T>0$.
\begin{proof}[\rm{\bf{Proof of Theorem~\ref{wellposed_1}.}}]
Let $(u_{0}$, $v_{0}$, $w_{0})\in B_{r}(H^s\times H^s\times H^s)$ be given. 
By Proposition~\ref{key_be} with $(\sigma_1,\sigma_2,\sigma_3)\in \{(\beta, \gamma, -\alpha), 
(-\gamma, \alpha, -\beta), (\alpha, -\beta, -\gamma)\}$ 
and Proposition~\ref{linear_est} with $\sigma \in \{\alpha, \beta, \gamma\}$, 
there exist $b'\in (0,\frac{1}{2})$ and constants $C_1$, $C_2$, $C_3>0$ such that for any $(u,v,w)\in B_R(X^{s,\frac{1}{2},1}_{\alpha,T}\times X^{s,\frac{1}{2},1}_{\beta,T}\times X^{s,\frac{1}{2},1}_{\gamma,T})$, 
we have
\[
\begin{split}
\|\Phi^{(1)}_{T,\alpha ,u_{0}}(w, v)\|_{X^{s,\frac{1}{2},1}_{\alpha,T}}&
\leq C_1\|u_{0}\|_{H^s} +CC_2C_3^2T^{1-2b'}\|w\|_{X^{s,\frac{1}{2},1}_{\gamma,T}}
\|v\|_{X^{s,\frac{1}{2},1}_{\beta,T}}\leq C_1r+CC_2C_3^2T^{1-2b'}R^2,\\
\|\Phi^{(1)}_{T,\beta ,v_{0}}(\overline{w}, u)\|_{X^{s,\frac{1}{2},1}_{\beta,T}}&
\leq C_1\|v_{0}\|_{H^s} +CC_2C_3^2T^{1-2b'}\|w\|_{X^{s,\frac{1}{2},1}_{\gamma,T}}
\|u\|_{X^{s,\frac{1}{2},1}_{\alpha,T}}\leq C_1r+CC_2C_3^2T^{1-2b'}R^2,\\
\|\Phi^{(2)}_{T,\gamma ,w_{0}}(u, \overline{v})\|_{X^{s,\frac{1}{2},1}_{\gamma,T}}&
\leq C_1\|w_{0}\|_{H^s} +CC_2C_3^2T^{1-2b'}\|u\|_{X^{s,\frac{1}{2},1}_{\alpha,T}}
\|v\|_{X^{s,\frac{1}{2},1}_{\beta,T}}\leq C_1r+CC_2C_3^2T^{1-2b'}R^2
\end{split}
\]
and
\[
\begin{split}
\|\Phi^{(1)}_{T,\alpha ,u_{0}}(w_{1}, v_{1})-\Phi^{(1)}_{T,\alpha ,u_{0}}(w_{2}, v_{2})\|_{X^{s,\frac{1}{2},1}_{\alpha,T}}
&\leq CC_2C_3^2T^{1-2b'}R\left( \|w_{1}-w_{2}\|_{X^{s,\frac{1}{2},1}_{\gamma,T}}+||v_{1}-v_{2}||_{X^{s,\frac{1}{2},1}_{\beta,T}}\right),\\
\|\Phi^{(1)}_{T,\beta ,v_{0}}(\overline{w_{1}}, u_{1})-\Phi^{(1)}_{T,\beta ,v_{0}}(\overline{w_{2}}, u_{2})\|_{X^{s,\frac{1}{2},1}_{\beta,T}}
&\leq CC_2C_3^2T^{1-2b'}R\left( \|w_{1}-w_{2}\|_{X^{s,\frac{1}{2},1}_{\gamma,T}}+\|u_{1}-u_{2}\|_{X^{s,\frac{1}{2},1}_{\alpha,T}}\right),\\
\|\Phi^{(2)}_{T,\gamma ,w_{0}}(u_{1}, \overline{v_{1}})-\Phi^{(2)}_{T,\gamma ,w_{0}}(u_{2}, \overline{v_{2}})\|_{X^{s,\frac{1}{2},1}_{\gamma,T}}
&\leq CC_2C_3^2T^{1-2b'}R\left( \|u_{1}-u_{2}\|_{X^{s,\frac{1}{2},1}_{\alpha,T}}+\|v_{1}-v_{2}\|_{X^{s,\frac{1}{2},1}_{\beta,T}}\right).
\end{split}
\]
Therefore if we choose $R>0$ and $T>0$ as
\[
R=6C_1r,\ CC_2C_3^2T^{1-2b'}R\le \frac{1}{4}
\]
then $\Phi$ is a contraction map on $B_R(X^{s,\frac{1}{2},1}_{\alpha,T}\times X^{s,\frac{1}{2},1}_{\beta,T}\times X^{s,\frac{1}{2},1}_{\gamma,T})$. 
This implies the existence of the solution of the system (\ref{NLS_sys}) and the uniqueness in the ball $B_R(X^{s,\frac{1}{2},1}_{\alpha,T}\times X^{s,\frac{1}{2},1}_{\beta,T}\times X^{s,\frac{1}{2},1}_{\gamma,T})$. 
The Lipschitz continuously of the flow map is also proved by similar argument. 
\end{proof} 
\section*{acknowledgements}
The first author is financially supported by JSPS KAKENHI Grant Number 17K14220
and Program to Disseminate Tenure Tracking System from the Ministry of Education, Culture, Sports, Science and Technology. The second author is supported by Grant-in-Aid for JSPS Research Fellow 16J11453.

\end{document}